\documentclass[sn-basic,Numbered]{sn-jnl}

\usepackage{apxproof}
\usepackage{graphicx}%
\usepackage{enumerate}%
\usepackage{multirow}%
\usepackage{amsmath,amssymb,amsfonts}%
\usepackage{bbm}%
\usepackage{mathrsfs}%
\usepackage[title]{appendix}%
\usepackage{xcolor}%
\usepackage{textcomp}%
\usepackage{manyfoot}%
\usepackage{booktabs}%
\usepackage{algorithm}%
\usepackage{algpseudocode}%
\usepackage{mathtools}
\usepackage{tabularx}
\usepackage{soul}
\usepackage{caption}
\usepackage{tikz}
\usetikzlibrary{positioning, fit, calc, backgrounds, tikzmark, decorations.pathreplacing}

\usepackage[labelfont=sf]{subcaption}
\usepackage[most]{tcolorbox}

\newtcolorbox{examplebox}[1][]{
    enhanced,
    colback=gray!5!white,  %
    colframe=black!80,     %
    boxrule=0.5pt,         %
    arc=0pt,               %
    outer arc=0pt,
    title=#1,              %
    fonttitle=\bfseries,   %
    coltitle=white,        %
    attach boxed title to top left={yshift=-2mm, xshift=5mm}, %
    boxed title style={colback=black!80} %
}

\theoremstyle{thmstyleone}%
\newtheoremrep{theorem}{Theorem}%

\newtheoremrep{lemma}{Lemma}%

\theoremstyle{thmstyletwo}%

\newcommand{\weakto}{\rightharpoonup}
\newcommand{\setto}{\rightrightarrows}
\newcommand{\zer}[1]{\mathrm{zer}\left( #1 \right)}
\newcommand{\R}{\mathbb R}

\newcommand{\Sp}{\mathbb{S}_+}
\newcommand{\I}{\mathrm{Id}}
\newcommand{\1}{\mathbbm{1}}
\newcommand{\norm}[1]{\left\lVert#1\right\rVert}
\newcommand{\HH}{\mathcal{H}}
\newcommand{\GG}{\mathcal{G}}
\newcommand{\Gy}{\bigoplus_{k \in \mathcal{K}}\GG_k} %
\newcommand{\BB}{\mathcal{B}}
\newcommand{\Hx}{\mathcal{H}_\text{x}}
\newcommand{\Hy}{\mathcal{H}_\text{y}}
\newcommand{\Bx}{\mathcal{B}_\text{x}}
\newcommand{\By}{\mathcal{B}_\text{y}}
\newcommand{\Hz}{\mathcal{H}_\text{z}}
\newcommand{\Gk}{\GG_k}
\renewcommand{\v}[1]{\mathbf{#1}}
\newcommand{\PA}{\v{P}}
\newcommand{\PAs}{\PA^*}

\newcommand{\PB}{\v{P}'}
\newcommand{\PBs}{(\PB)^*}

\newcommand{\ZA}{\tilde{\v{Z}}}
\newcommand{\WA}{\tilde{\v{W}}}
\newcommand{\MA}{\tilde{\v{M}}}
\newcommand{\MT}{\MA^*}
\newcommand{\DA}{\tilde{\v{D}}}
\newcommand{\UA}{\tilde{\v{U}}}
\newcommand{\LA}{\tilde{\v{L}}}
\newcommand{\KA}{\tilde{\v{K}}}
\newcommand{\kA}{\mathcal{K}}

\newcommand{\kBj}[1]{\mathcal{K}_{#1}'}
\newcommand{\QA}{\tilde{\v{Q}}}
\newcommand{\RA}{\v{R}}
\newcommand{\RAs}{\RA^*}
\newcommand{\RAi}[1]{R_{#1}}
\newcommand{\RB}{\v{R}'}
\newcommand{\RBj}[1]{R_{#1}'}
\newcommand{\RBs}{(\RB)^*}
\newcommand{\xb}{\v{x}'}

\newcommand{\yb}{\v{y}'}
\newcommand{\ybk}{\v{y}'}

\DeclareMathOperator{\prox}{prox}
\DeclareMathOperator{\pos}{pos}
\DeclareMathOperator{\tril}{tri}
\DeclareMathOperator{\diag}{diag}
\DeclareMathOperator{\nullspace}{null}
\DeclareMathOperator{\range}{ran}
\DeclareMathOperator{\dom}{dom}
\DeclareMathOperator{\gra}{graph}

\DeclareMathOperator{\ri}{ri}
\DeclareMathOperator*{\minimize}{minimize}
\DeclareMathOperator*{\argmin}{argmin}

\newenvironment{notation*}
  {\par\vspace{\abovedisplayskip}\noindent
   \tabularx{\columnwidth}{>{$}l<{$} @{${}:{}$} >{\raggedright\arraybackslash}X}}
  {\endtabularx\par\vspace{\belowdisplayskip}}

\begin{document}

\title{Coupled Adaptable Backward-Forward-Backward Resolvent Splitting Algorithms (CABRA)}

\author*{\fnm{Peter} \sur{Barkley}}\email{pbarkley@usna.edu}
\author{\fnm{Robert L} \sur{Bassett}}\email{robert.bassett@nps.edu}

\affil[1]{\orgdiv{Mathematics Department}, \orgname{United States Naval Academy}, \orgaddress{\street{572C Holloway Road}, \city{Annapolis}, \postcode{21402}, \state{MD}, \country{USA}}}
\affil[2]{\orgdiv{Operations Research and Applied Mathematics Departments}, \orgname{Naval Postgraduate School}, \orgaddress{\street{1 University Circle}, \city{Monterey}, \postcode{93943}, \state{CA}, \country{USA}}}

\abstract{
We present a novel matrix-parametrized frugal splitting algorithm which finds a zero of a sum of maximal monotone and cocoercive operators composed with linear selection operators, which models a large class of problems in which the operators take only a subset of decision variables as their input.
We also develop a semidefinite programming framework for selecting matrix parameters and demonstrate its use for designing matrix parameters which provide beneficial diagonal scaling, allow parallelization, and adhere to a given communication structure.
We show that making use of the structure of the linear selection operators in this way improves convergence in numerical experiments. 
Even when the selection operators are the identity, we can enhance convergence by using the matrix parameters to provide appropriately chosen diagonal scaling.
We conclude by demonstrating the applicability of this algorithm to multi-stage stochastic programming,
outlining a decentralized approach to the relaxed stochastic weapon target assignment problem which splits over the source nodes and has low data transfer and memory requirements.
}

\maketitle

\section{Introduction}
\label{cab:Sec:Introduction}
This work presents a novel decentralized and largely parallelizable family of resolvent (or proximal) splitting algorithms for solving convex optimization problems involving the sums of multiple convex, possibly $L$-smooth, functions under composition with linear selection operators.
We consider the problem
\begin{equation}
    \label{cab:argmin_problem}
    \text{find} \; y \in \R^p \quad \text{such that} \quad y \in \argmin_{x} \sum_{i=1}^n f_i(\RAi{i} x) + \sum_{j=1}^m g_j(\RBj{j} x)
\end{equation}
where each $f_i\colon \R^{p_i} \to \R \cup \{\infty\}$ is a lower semi-continuous, convex, and proper function, each $g_j\colon \R^{p_j'} \to \R $ is convex and $\frac{1}{\beta_j}$-smooth, and each $\RAi{i}$ ($\RBj{j}$) is a linear selection matrix from $\R^p$ to $\R^{p_i}$ ($\R^{p_j'}$).
That is, each selection operator $\RAi{i}$ or $\RBj{j}$ is an identity matrix on $\R^p$ with rows removed for each variable which is not used by $f_i$ or $g_j$, respectively.
This problem encompasses a wide range of problems in machine learning, statistics, federated learning, operations research, finance, optimal control, and signal processing \cite{combettes2011proximal,parikh2014proximal,combettes2023geometry,ryu2022large,aakerman2025splitting,ryu2020uniqueness}.
These include group lasso, multistage stochastic programming, sensor network localization, and the continuous relaxation of the weapon target assignment problem \cite{barkley2025decentralizedsensornetworklocalization,eckstein2022async}.
Under standard regularity conditions, problem \eqref{cab:argmin_problem} is also a subset of a larger class of monotone inclusion problems given by
\begin{equation}
    \label{cab:zero_in_coupled_sum}
    \text{find}\; y \in \bigoplus_{k = 1}^p\GG_k \;\text{such that}\; 0 \in \sum_{i = 1}^n \RAi{i}^* A_i \RAi{i} y + \sum_{j = 1}^m (\RBj{j})^* B_j \RBj{j} y
\end{equation}
where each $A_i$ is a maximal monotone operator, each $B_j$ is a $\beta_j$-cocoercive operator, each $\GG_k$ is a real Hilbert space, and each $\RAi{i}$ and $\RBj{j}$ now selects a subset of $(y_1, \dots, y_p)$.
With $Ai = \partial f_i$, $B_j = \nabla g_j$, and $\GG_k = \R$, we recover problem \eqref{cab:argmin_problem}.
Problem \eqref{cab:zero_in_coupled_sum}, however, allows us to consider saddle point problems and other more general problems in addition to those contained in problem \eqref{cab:argmin_problem}.

Following the characterization of frugal resolvent splittings (FRS) with minimal lifting in \cite{ryu2020uniqueness} and the identification of an $n$-operator frugal resolvent splitting with minimal lifting in \cite{malitsky2023resolvent} and graph-based $n$-operator splittings in \cite{bredies2024graph, tam2023frugal}, FRS algorithms with minimal lifting have been extended to the sums of monotone and cocoercive operators in a wide variety of ways \cite{morin2024frugal,aragon2023distributed,aragon2024forward,aakerman2025splitting,dao2025general,zong2023solving}.
The application of these types of algorithms to summations with arbitrary compositions with bounded linear operators (as \cite{bui2022multivariate}, for example, does without minimal lifting) remains an open question, however \cite{aragon2023primal}.
When these operators are selection operators, failing to account for the structure of the composition leads these algorithms to work in a lifted space which can be much larger than the spaces used by the functions themselves, requiring additional unnecessary computation and memory, and slowing convergence in practice.
In the remainder, we will present a family of algorithms which fills this gap in the case of linear selection operators, solving both \eqref{cab:argmin_problem} and \eqref{cab:zero_in_coupled_sum} with direct proximal and gradient calls on the functions.
We confine our initial presentation to the context of \eqref{cab:argmin_problem} for clarity, providing a full description and proof of the general case in the appendix, and use an expanded variant of the algorithms, which we show is equivalent to the reduced form associated with minimal lifting.

We next give a high level summary of our contributions, with a thorough introduction to follow in later sections.
For simplicity, we develop the algorithmic approach incrementally, beginning with case where $m=0$.
We define a family of proximal splitting algorithms governed by the PSD block matrices $\ZA$ and $\WA$, which have blocks $\ZA_{ii'}, \WA_{ii'} \in \R^{{p_{i} \times p_{i'}}}$ for $i, i' \in \{1, \dots, n\}$. 
We denote by $\DA_i$ the $i$th block diagonal matrix in $\ZA$ so that $\DA_i = \ZA_{ii} \in \mathbb{S}_{++}^{p_{i}}$.
We refer to this family as Coupled Adaptable Resolvent Splitting Algorithms (CARA), and they are defined by the iterations:
\begin{subequations}
    \label{cab:cara}
\begin{align}
    \v{x}^\nu_1 &= \prox_{\alpha \DA_1^{-1} f_1}\left(\DA_1^{-1} \v{v}^\nu_1\right) \\
    \v{x}^\nu_i &= \prox_{\alpha \DA_i^{-1} f_i}\left(\DA_i^{-1} \left[\v{v}^\nu_i - 2 \sum_{i' < i} \ZA_{ii'}\v{x}^\nu_{i'}\right] \right)  \quad i = 2, \dots, n  \\
    \v{v}^{\nu+1} &= \v{v}^\nu - \gamma_\nu \WA \v{x}^\nu
\end{align}
\end{subequations}
for $\alpha > 0$ and $\gamma_\nu \in (0, 2)$.

Then, in the full case of problem \eqref{cab:argmin_problem} where we include smooth functions, we require two additional block matrices, $\QA$ and $\KA$, with block matrix entries $\QA_{ij} \in \R^{p_i \times p_j'}$ and $\KA_{ji}  \in \R^{p_j' \times p_i}$ for $i \in \{1, \dots, n\}$ and $j \in \{1, \dots, m\}$, and must restrict the valid domain of $\alpha$ and $\gamma_\nu$. 
We then define the family of Coupled Adaptable Backward Forward-Backward Algorithms (CABRA) as
\begin{align*}
    \v{x}^\nu_1 &= \prox_{\alpha \DA_1^{-1} f_1}\left(\DA_1^{-1} \v{v}^\nu_1\right) \\
    \v{x}^\nu_i &= \prox_{\alpha \DA_i^{-1} f_i}\left(\DA_i^{-1} \left[\v{v}^\nu_i - 2 \sum_{i' < i} \ZA_{ii'}\v{x}^\nu_{i'} + \alpha \sum_{j=1}^m \QA_{ij} \nabla g_j (\sum_{i' < i } \KA_{ji'} \v{x}^\nu_{i'})\right] \right)  \quad i = 2, \dots, n  \\
    \v{v}^{\nu+1} &= \v{v}^\nu - \gamma_\nu \WA \v{x}^\nu.    
\end{align*}
We emphasize, as a point of distinction from previous frugal resolvent splitting algorithms, that each $\ZA_{ii'}$, $\QA_{ij}$, $\KA_{ji}$, and $\DA_i$ is a matrix rather than a scalar multiplier.
We will further define these matrices and the criteria they must meet for convergence in the following sections, but we note that they can be directly selected based on the graph of a connected communication network, as well as by a number of more sophisticated methods, and can be chosen to allow decentralized execution with most (but not all) of the proximal operations executing in parallel.

\subsection*{Notation}
We rely on the following definitions and terminology throughout this work.
We refer to a set-valued operator $A\colon \HH \setto \HH$ for some real Hilbert space $\HH$ as \textit{monotone} if $\langle x-y, u-v\rangle \geq 0$ for all $(x, u)$ and $(y, v)$ in $\gra(A) = \{(x, u) : u \in Ax\}$.
It is maximal monotone if its graph is not properly contained in the graph of any other monotone operator.
We use the notation $J_A = (\I + A)^{-1}$ to refer to the \textit{resolvent} of $A$, and note that for a convex function $f$ with subdifferential $\partial f$, $J_{\partial f}(x) = \prox_{f}(x)$, where 
\begin{equation*}
    \prox_{f}(z) = \argmin_x f(x) + \frac{1}{2}\norm{x-z}^2.
\end{equation*}
An operator $A$ is $\gamma$-averaged nonexpansive if \[\norm{u - v}^2 \leq \norm{x-y}^2 - \frac{1-\gamma}{\gamma}\norm{(x - u) - (y - v)}^2\] for all $(x, u)$ and $(y, v)$ in $\gra(A)$.
We also note that the resolvent of a maximal monotone operator is firmly nonexpansive and single-valued with full domain \cite{bauschke_combettes}.
A differentiable function $f$ is said to be $\frac{1}{\beta}$-smooth if $\norm{\nabla f(x) - \nabla f(y) } \leq \frac{1}{\beta}\norm{x-y}$.
An operator $A$ is said to be $\beta$-cocoercive if $\langle x-y, Ax-Ay\rangle \geq \beta \norm{Ax - Ay}^2$. 
We note that cocoercive operators are necessarily Lipschitz and single-valued.
By the Baillon-Haddad theorem, the gradient of a convex $\frac{1}{\beta}$-smooth function is $\beta$-cocoercive \cite{baillon1977quelques}.

We also use the term {graph} to refer to a collection of nodes and edges $G = (V, E)$, and {weighted} graph for a graph with vector of weights corresponding to its edges. 
For a weighted graph $G$ on $n$ nodes, we define the graph Laplacian of $G$ as a matrix $Z \in \Sp^{n}$ such that $Z=D-A$ where $D$ is the weighted diagonal degree matrix and $A$ is the weighted adjacency matrix. %
We write $\Sp^n$ for the set of symmetric $n \times n$ positive semi-definite (PSD) matrices, $\succeq$ for the generalized inequality on the PSD cone, and $\otimes$ for the Kronecker product.
We define the matrix $\tril(Z)$ as the strict lower triangle of square matrix $Z$, and overload the diagonal notation so that $\diag(Z)$ returns the diagonal matrix with the diagonal of a square matrix $Z$ and zeros elsewhere, and $\diag(x)$ returns the diagonal matrix with vector $x$ on the diagonal when $x$ is a vector.
We let the position function $\pos(I, k)$ denote the 1-based index of element $k$ in the tuple $I$.

\section{Background}
\label{cab:Sec:Background}

In \cite{ryu2020uniqueness}, Ryu defines a \textit{resolvent splitting} as a splitting constructed only with resolvents, scalar multiplication, and addition, and calls a resolvent splitting \textit{frugal} if it evaluates each resolvent only once in each iteration.
He also constructs a frugal three-operator resolvent splitting on $\HH$ exhibiting minimal lifting, in that it lifts to $\HH^{n-1}=\HH^2$ rather than $\HH^3$. 
In \cite{malitsky2023resolvent}, the authors generalize this to find a frugal resolvent splitting (FRS) with minimal lifting over $n$ operators.
This algorithm is generalized further into a family of FRS algorithms defined by a set of matrix parameters in \cite{tam2023frugal} and, in \cite{bredies2024graph}, a family of FRS algorithms defined by graphs (and their Laplacians).
The work in \cite{bassett2024optimaldesignresolventsplitting} introduces a semi-definite programming (SDP)-based approach to matrix parameter/weighted graph selection.

The analysis in \cite{morin2024frugal} extends these FRS approaches to the summation of a set of maximal monotone and cocoercive operators. 
The authors of \cite{aragon2024forward} build graph-based frugal splittings with minimal lifting in a similar setting, requiring the number of cocoercive operators to be less than the number of maximal monotone operators, and imposing a number of subgraph requirements on the unweighted graphs associated with the parameters.
In \cite{aakerman2025splitting}, the authors extend this broadly to define not only the sufficient conditions for the matrix parametrization, but also prove that these conditions are necessary for any frugal splitting operator with minimal lifting and the fixed-point encoding property, with a similar algorithm recently given by \cite{dao2025general} with application to a broader class of problems.

A number of other proximal algorithms solve similar problems over the sums of monotone and cocoercive operators with linear compositions, but without minimal lifting, including \cite{johnstone2021single,pgextra,nids,bui2022multivariate}.
Among these, PG-EXTRA \cite{pgextra} and NIDS \cite{nids} allow decentralized execution over an arbitrary network, but do not account for selection operators, forcing implementations to operate on the composite operators $(\RAi{i})^* \partial f_i \RAi{i}$ and $(\RBj{j})^* \nabla g_j \RBj{j}$.
The algorithm in \cite{johnstone2021single} requires $n=m$ and $\RAi{i} = \RBj{j}$ for $i=j$, and requires a centralized computation in each iteration.
Likewise, \cite{bui2022multivariate} extends proximal splitting to a broad swathe of problems, but still requires a centralized computation in each iteration.
Recent work in \cite{dao2025primal} also considers the sum of monotone and cocoercive operators with linear compositions, but requires at least two monotone operators to accept the full space of the decision variable without composition, leading, in some cases, to greater memory and computation requirements.

In the context of stochastic programming, a number of algorithms solve a specific sub-class of problem \eqref{cab:argmin_problem} where the selection operators each select the variables associated with a specific scenario. 
These include Progressive Hedging \cite{rockafellar1991scenarios}, Projective Hedging \cite{eckstein2009general}, and a block-activated decomposition algorithm \cite{bui2025block}, all of which permit the solution of all proximal operators in parallel, but still require centralized computation in each iteration, lifting above the minimal lifting, and do not explicitly permit forward steps for differentiable operators.
Our work addresses this gap by reducing the requiring lifting and implementing forward steps and a matrix-parametrized approach which permits decentralized, near-parallel execution.

\section{Coupled Adaptable Resolvent Splitting Algorithms}
\label{cab:Sec:CARA}

We now introduce the CARA family of algorithms, beginning with a description of the lifted space in which they operate, a description of the matrix parameters with some examples, and a sketch of the convergence proof for the algorithms, with detailed proofs provided in the appendix.

We begin by defining some notation for the linear selection operators.
Let $\mathcal{K} = \{1, \dots, p\}$ and let $\kA_i$ be a strictly increasing sequence of the elements of $y \in \R^p$ selected by each operator $\RAi{i}$, letting $p_i = |\kA_i|$.
We then define $\RAi{i}\colon \R^p \to \R^{p_i}$ as an identity matrix in $\R^p$ with rows deleted for the entries which are not in $\kA_i$.
We let $\mathcal{I} = \{1, \dots, n\}$ and $\mathcal{I}_k$ be the strictly increasing sequence of the indices of the functions $f_i$ that take $y_k$ as a part of their argument (i.e. with $k \in \kA_i$), and let $n_k = |\mathcal{I}_k|$ give the number of functions $f_i$ taking $y_k$ as part of their argument.
We assume that $n_k \geq 2$ throughout, so that every $y_k$ is an input to at least two functions $f_i$.

As splitting algorithms, the approaches we consider create copies of the elements of the decision variable $y \in \R^p$, but unlike many proximal splitting methods, they do not copy $y$ in its entirety for each function $f_i$, but only create copies of exactly the subset of variables needed for each $f_i$.
These copies can be grouped by the function $f_i$ using them, or by the scalar variable being replicated.

Grouping by function, we write $\v{x}_i \in \R^{p_i}$, and use the notation $\v{x}_{ik}$ to give the entry in $\v{x}_i$ associated with $y_k$ for $k \in \kA_i$, rather than $\v{x}_{i, \pos(\kA_i, k)}$.
We then define the lifted variable $\v{x} = (\v{x}_i)_{i \in \mathcal{I}} \in \prod_{i=1}^n \R^{p_i}$.
This lifted space contains $n_k$ variables associated with each $y_k$.

Grouping by variable instead, we define $\v{y}_k \in \R^{n_k}$, and use the notation $\v{y}_{ki}$ to give the entry in $\v{y}_k$ associated with function $f_i$, rather than $\v{y}_{k, \pos(\mathcal{I}_k, i)}$.
We then form the lifted vector $\v{y} = (\v{y}_k)_{k \in \mathcal{K}} \in \prod_{k=1}^p \R^{n_k}$.
\begin{examplebox}[Function and Variable Centric Lifted Spaces]
For example, consider a problem with $p=3$ and $n=3$, and 
\begin{align*}
\kA_1 &= (1,2,3), \;&\mathcal{I}_1 &= (1,3), \\
\kA_2 &= (2,3), \;  &\mathcal{I}_2 &= (1,2), \\
\kA_3 &= (1,3), \;  &\mathcal{I}_3 &= (1,2,3),
\end{align*}
so that $f_1$ uses all variables, $f_2$ uses $y_2$ and $y_3$, and $f_3$ uses $y_1$ and $y_3$.
We then have the following lifted variables: 
\begin{equation*}
\hspace*{-2.5cm}
\begin{aligned}
\text{Function-centric:} \quad \v{x} &= \Bigg( 
\underbrace{\tikzmarknode{x11}{\color{blue}\v{x}_{11}}, \tikzmarknode{x12}{\color{red}\v{x}_{12}}, \tikzmarknode{x13}{\color{teal}\v{x}_{13}}}_{\v{x}_1}, \;\; 
\underbrace{\tikzmarknode{x22}{\color{red}\v{x}_{22}}, \tikzmarknode{x23}{\color{teal}\v{x}_{23}}}_{\v{x}_2}, \;\; 
\underbrace{\tikzmarknode{x31}{\color{blue}\v{x}_{31}}, \tikzmarknode{x33}{\color{teal}\v{x}_{33}}}_{\v{x}_3} 
\tikzmarknode{topR}{\Bigg)} \\[3em] %
\text{Variable-centric:} \quad \v{y} &= \Bigg( 
\underbrace{\tikzmarknode{y11}{\color{blue}\v{y}_{11}}, \tikzmarknode{y13}{\color{blue}\v{y}_{13}}}_{\v{y}_1}, \;\; 
\underbrace{\tikzmarknode{y21}{\color{red}\v{y}_{21}}, \tikzmarknode{y22}{\color{red}\v{y}_{22}}}_{\v{y}_2}, \;\; 
\underbrace{\tikzmarknode{y31}{\color{teal}\v{y}_{31}}, \tikzmarknode{y32}{\color{teal}\v{y}_{32}}, \tikzmarknode{y33}{\color{teal}\v{y}_{33}}}_{\v{y}_3}
\tikzmarknode{botR}{\Bigg)}.
\end{aligned}
\end{equation*}

\begin{tikzpicture}[remember picture, overlay, >=stealth]
    \tikzset{
        link/.style={<->, thick, shorten >=2pt, out=-90, in=90, looseness=1.2, opacity=0.75}
    }

    \draw[blue, link] ([yshift=-14pt]x11.south) to (y11.north);
    \draw[blue, link] ([yshift=-14pt]x31.south) to (y13.north);

    \draw[red, link] ([yshift=-14pt]x12.south) to (y21.north);
    \draw[red, link] ([yshift=-14pt]x22.south) to (y22.north);

    \draw[teal, link] ([yshift=-14pt]x13.south) to (y31.north);
    \draw[teal, link] ([yshift=-14pt]x23.south) to (y32.north);
    \draw[teal, link] ([yshift=-14pt]x33.south) to (y33.north);

    \draw[decorate, decoration={brace, amplitude=5pt}, thick]
        ([xshift=12pt]topR.south east) -- ([xshift=12pt]botR.north east)
        node[midway, right=6pt] {Permutation $P$};
\end{tikzpicture}

\end{examplebox}
\noindent We then define a permutation matrix $\PA$ which maps from the ordering in $\v{y}$ to that in $\v{x}$, with \[\PA \colon \prod_{k=1}^p \R^{n_k} \to \prod_{i=1}^n \R^{p_i} ,\] so that for $\v{x} = \PA\v{y}$, we have $\v{x}_{ik} = \v{y}_{ki}$.
\subsection*{Matrices}
We choose two sets of $p$ matrices $Z_k, W_k \in \Sp^{n_k}$ for $k \in \mathcal{K}$, which operate on $\v{y}_k$ in each lifted space, so that the entry in row $r= \pos(\mathcal{I}_k, i)$ and column $c=\pos(\mathcal{I}_k, i')$ scales the $y_k$ entry being sent from the proximal operator of $f_{i'}$ to that of $f_i$. 
We then form block diagonal matrices
\begin{equation}
    \label{cab:cara_ZWdiag}
    \v{Z} = \begin{pmatrix}
        Z_1 &     & & \\
            & Z_2 & & \\
            &     & \ddots & \\
            &     &       & Z_p 
    \end{pmatrix} \; \text{and}\; \v{W} = \begin{pmatrix}
        W_1 &     & & \\
            & W_2 & & \\
            &     & \ddots & \\
            &     &       & W_p 
    \end{pmatrix}
\end{equation}
which correspond with a variable-centric ordering in the lifted space.
Each block in the block-diagonal matrix gives the exchange (via its sparsity pattern) and scaling (via non-zero values) of variables among operators that take that variable as an input.
In order to operate on the function-centric lifted variables, we define the permuted matrices \begin{equation}
    \label{cab:cara_ZWtilde}
\ZA = \PA \v{Z} \PAs \;\text{and}\; \WA = \PA \v{W} \PAs.
\end{equation}
We then let $\ZA_{ii'}$ refer to the block of $\ZA$ mapping the outputs of the proximal operator of $f_{i'}$ to the input vector for $\prox_{f_i}$, with $\ZA_{ii'} \in \R^{p_i \times p_{i'}}$.
We also let 
\begin{subequations}
    \label{cab:cara_DL}
\begin{align}
\DA_i &= \ZA_{ii} \quad &&\forall i \in \mathcal{I},\; \text{and}\; \\
\LA_{ii'} &= \begin{cases}
- \ZA_{ii'} \quad &i>i' \\
0           \quad &\text{otherwise}
\end{cases}\quad &&\forall i, i' \in \mathcal{I}.
\end{align}
\end{subequations}

Having outlined the construction of the matrices $\ZA$ and $\WA$ based on matrices $Z_k$ and $W_k$, we now consider the conditions on the $Z_k$ and $W_k$ matrices under which Algorithm \ref{cab:cara_n_alg} converges.
These are similar to the conditions placed on $Z$ and $W$ in \cite{bassett2024optimaldesignresolventsplitting}, which have direct analogs in \cite{tam2023frugal,bredies2024graph,aakerman2025splitting,dao2025general}, namely, for all $k \in \mathcal{K}$,
\begin{subequations}
    \label{cab:cara_reqs}
\begin{align}
    &Z_k \succeq W_k \label{cab:Z_succeq_W}\\
    &\nullspace(W_k) = \text{span}(\1_{n_k})\label{cab:W_connect}\\
    &\nullspace(Z_k) = \text{span}(\1_{n_k})\label{cab:Z_connect}.
\end{align} 
\end{subequations}
A critical requirement for the convergence of CARA to a point where there exist subgradients $\v{w}_i^* \in \partial f_i(\v{x}_i^*)$ such that $0 = \sum_{i=1}^n \RAi{i}^* \v{w}_i^*$ is that $0 = \sum_{i=1}^n \RAi{i}^* \v{v}_i^\nu$ for each $\v{v}^\nu$.
The null space requirement in \eqref{cab:W_connect} guarantees that the range of $\WA$ is exactly the space of vectors with this property.
We then guarantee that every $\v{v}^\nu$ remains in this space by requiring
\begin{equation}
    \label{cab:v_req}
    0 = \sum_{i=1}^n \RAi{i}^* \v{v}_i^0,
\end{equation}
which is equivalent to requiring $\sum_{i \in \mathcal{I}_k} \v{v}_{ik}^0 = 0$ for all $k$.

Including these requirements, we can succinctly state the CARA algorithm in full as given by Algorithm \ref{cab:cara_n_alg}, letting $\v{A} = (\partial f_1, \dots, \partial f_n)$.
\begin{algorithm} 
    \caption{CARA (expanded form)}\label{cab:cara_n_alg} 
    \begin{algorithmic}[1] 
    \Require $\alpha \in (0,\infty)$; $\gamma_\nu \in (0, 2)$; $\DA$, $\WA$, $\LA$ satisfying \eqref{cab:cara_ZWdiag}, \eqref{cab:cara_ZWtilde}, \eqref{cab:cara_DL}, and \eqref{cab:cara_reqs}; $\v{v}^0$ satisfying \eqref{cab:v_req}.
    \State $\nu \gets 0$
    \Repeat
      \State $\v{x}^\nu = J_{\alpha \DA^{-1}\v{A}}\left(\DA^{-1}\left[\v{v}^\nu + 2\LA\v{x}^\nu \right]\right)$ \label{cab:cara_n_itr1}
      \State $\v{v}^{\nu+1} = \v{v}^\nu - \gamma_\nu \WA\v{x}^\nu$ \label{cab:cara_n_itr2}
      \State $\nu \gets \nu+1$
    \Until{convergence}
\end{algorithmic}
\end{algorithm}

Our next theorem establishes the convergence of the CARA algorithm as a special case of Theorem \ref{cab:main_theorem} in the appendix, which proves the convergence of the algorithm over an arbitrary set of Hilbert spaces.
\begin{theorem}
    \label{cab:cara_theorem}
Let $(\v{x}^\nu)$ and $(\v{v}^\nu)$ be the sequences of iterates produced by Algorithm \ref{cab:cara_n_alg} applied to problem \eqref{cab:argmin_problem} with $m=0$ for a given sequence of step sizes $\gamma_\nu \in (0, 2)$ with $\sum_\nu \gamma_\nu(2-\gamma_\nu) = \infty$. 
If the set of solutions of \eqref{cab:argmin_problem} is non-empty and there exists some $x \in \R^p$ such that $ \RAi{i} x \in \ri(\dom f_i)$ for all $i \in \mathcal{I}$, then 
\begin{enumerate}
    \item $(\v{x}^\nu)$ converges to a point $\v{x}^*$ such that $\v{x}^*_i =\RAi{i} y^*$ where $y^*$ is a solution to problem \eqref{cab:argmin_problem}.
    \item $(\v{v}^\nu)$ converges to a point $\v{v}^*$ such that $\v{v}_i^* = \alpha \v{w}_i^* + (\ZA_{ii} + 2\sum_{i' < i} \ZA_{ii'})\v{x}^*_i$ for $\v{w}_i^* \in \partial f_i(\v{x}_i^*)$. %
    \item $\begin{aligned}
    \sum_{k \in \mathcal{K}}\sum_{i \in \mathcal{I}_k}\norm{\v{x}_{ik}^\nu - \bar{y}_{k}^\nu}^2 = o\left(\frac{1}{\nu +1}\right)
    \end{aligned}$
    \item $\begin{aligned}
        \norm{\sum_{i \in \mathcal{I}}\RAi{i}^* \v{w}_i^\nu }^2 = o\left(\frac{1}{\nu +1}\right)
        \end{aligned}$
\end{enumerate}
where $\bar{y}^\nu_{k} = \frac{1}{n_k}\sum_{i \in \mathcal{I}_k} \v{x}_{ik}^\nu$ gives the average value the elements associated with $y_k$ across $\v{x}_i^\nu$ for $i \in \mathcal{I}_k$.
\end{theorem}
\begin{proof}
    Lemma \ref{cab:Lemma:expanded_equivalent} establishes the equivalence of the expanded form of the algorithm given in Algorithm \ref{cab:cara_n_alg} with the reduced form given by Algorithm \ref{cab:cabra_d_alg}, which we consider in Theorem \ref{cab:main_theorem}.
    The constraint qualification on the relative interior of each function establishes the equivalence of problem \eqref{cab:argmin_problem} with problem \eqref{cab:zero_in_coupled_sum} when $m=0$ and $\Gk = \R$.
    Theorem \ref{cab:cara_theorem} then follows directly from the proof given in Theorem \ref{cab:main_theorem} in the appendix.
    For completeness, we briefly sketch the argument here.
    Lemma \ref{cab:Lemma:fixed_point_corresponds_coupled} uses assumptions \eqref{cab:W_connect} and \eqref{cab:Z_connect} to show that fixed points of Algorithm \ref{cab:cabra_d_alg} correspond with solutions to problem \eqref{cab:zero_in_coupled_sum}.
    We then show in Lemma \ref{cab:Lemma:Nonexpansive_coupled} that assumption \eqref{cab:Z_succeq_W} and the monotonicity of the subdifferentials of the convex functions $f_i$ guarantee that the algorithm is an averaged-nonexpansive fixed point iteration.
    This leads to the convergence result in Theorem \ref{cab:main_theorem}, in which we establish the weak convergence of the iterates over real Hilbert spaces.
    Applying the results in Theorem \ref{cab:main_theorem} to problem \eqref{cab:argmin_problem} with $y \in \R^p$ then yields the first and second results. 
    In the context of CARA, with $m=0$, the range of convergent step sizes $\gamma_\nu$ becomes $(0,2)$ and the range of the scaling parameter $\alpha$ becomes $(0, \infty)$.
    Applying the convergence rate results in Theorem \ref{cab:convergence_rate} with $m=0$ provides the final results.
\end{proof}

The requirements given by \eqref{cab:cara_reqs} permit a substantial design space for these matrices.
These designs can be guided by the fact that each PSD matrix with its null space spanned by the ones vector is the graph Laplacian of some connected weighted graph. 
The graphs associated with $Z_k$ and $W_k$ can, in fact, be interpreted as the communication network for element $y_k$ with the nodes formed by the proximal operators in the algorithm.
The $Z_k$ graph governs the information sharing within the proximal computation step, and the $W_k$ graph governs the information sharing required for the $\v{v}$ update step.
One of the most straightforward choices of these matrices, then, is to choose $Z_k = W_k$ to be the graph Laplacian of the complete graph on $n_k$ nodes \eqref{cab:complete}, so that
\begin{equation}
    \label{cab:complete}
Z_k = W_k = \begin{pmatrix}
    n_k - 1 & -1 & \dots & -1 \\
    -1 & n_k - 1 &       & -1 \\
    \vdots & & \ddots & \vdots \\
    -1 & -1 & \dots & n_k - 1
\end{pmatrix} = n_k \I_{n_k} - \1_{n_k \times n_k}.
\end{equation}
Bipartite graphs provide another set of useful matrices, particularly the balanced bipartite graph, and the star graph, with the graph Laplacian of the star graph centered at the first node given by 
\begin{equation}
    \label{cab:star}
Z_k = W_k = \begin{pmatrix}
    n_k - 1 & -1 & \dots & -1 \\
    -1      & 1 &    0   & 0 \\
    \vdots  & 0  & \ddots & \vdots \\
    -1     & 0 & \dots & 1
\end{pmatrix} = \begin{pmatrix}
    n_{k} -1 & -\1_{n_k - 1}^\intercal \\
    -\1_{n_k -1} & \I_{n_k - 1}
\end{pmatrix}.
\end{equation}
These bipartite graphs, which we use in Section \ref{cab:Sec:Experiments}, can allow the members of each partition to execute in parallel, unlike the complete graph, which requires serial execution.
In the case of the balanced bipartite graph on an even number of nodes, if all of the $Z_k$ matrices only link proximal operators in the first half to those in the second, this means that we compute the first half of the proximal operators in parallel, followed by the second half of the proximal operators in parallel. 
For the star graph matrix given by \eqref{cab:star}, if each $Z_k$ uses the same proximal operator as the center of the graph, CARA executes by computing the first proximal operator, and then computing all of the remaining operators in parallel.

When $\RAi{1} = \I$, so that every element of $y$ is an argument to $f_1$, CARA on the star graph given by \eqref{cab:star} computes the proximal operators of $f_2, \dots f_n$ in parallel as given by
\begin{align*}
    \v{x}^{\nu}_1 &= \prox_{\alpha \DA_1^{-1} f_1}\left(\DA_1^{-1} \v{v}^{\nu}_1\right) &\\
    \v{x}^{\nu}_i &= \prox_{\alpha f_i}\left(\v{v}^{\nu}_i + 2 \RAi{i} \v{x}^{\nu}_1\right) &\quad \text{in parallel for}\; i=2, \dots, n\\
    &\!\left. \begin{aligned} 
        \mathllap{\v{v}^{\nu+1}_1} &= \v{v}^{\nu}_1 - \gamma \left(\DA_1 \v{x}^{\nu}_1 - \sum_{i=2}^n \RAi{i}^* \v{x}^{\nu}_i\right) \\
        \mathllap{\v{v}^{\nu+1}_i} &= \v{v}^{\nu}_i - \gamma\left(\v{x}^{\nu}_i - \RAi{i} \v{x}^{\nu}_1\right)
    \end{aligned}\right\} &\quad \text{in parallel for}\; i=1, \dots, n
\end{align*}
where $\DA_1 = \diag((n_1 - 1, n_2 - 1, \dots, n_p - 1))$.
Notably, the $\v{v}_i$ update for $i = 2, \dots, n$ only requires information from the first proximal operator. 
At no point do the spokes on the star graph need to communicate directly.

Similarly, with $\RAi{1} = \I$ as before, the algorithm on the complete graph matrices \eqref{cab:complete} is given by
\begin{align*}
    \v{x}^{\nu}_1 &= \prox_{\alpha \DA_1^{-1} f_1}\left(\DA_1^{-1} \v{v}^{\nu}_1\right) & \\
    \v{x}^{\nu}_i &= \prox_{\alpha \DA_i^{-1} f_i}\left(\DA_i^{-1}\left[ \v{v}^{\nu}_i + 2 \RAi{i} \sum_{i' < i}\RAi{i'}^* \v{x}^{\nu}_{i'}\right]\right) &\text{for}\; i = 2, \dots, n\\
    \v{v}^{\nu+1}_i &= \v{v}^{\nu}_i - \gamma\left(\DA_i \v{x}^{\nu}_i - \RAi{i} \sum_{i' \neq i} \RAi{i'}^* \v{x}^{\nu}_{i'} \right)  &\text{in parallel for}\; i=1, \dots, n
\end{align*}
where $\DA_1$ is as before, and $\DA_i = \diag((n_k - 1)_{k \in \kA_i})$.
We will discuss additional methods for selecting $Z_k$ and $W_k$ in section \ref{cab:Sec:Opt}.

We further note that, with $\RAi{i} = \I$ for all $i$, $Z_k = Z$, and $W_k = W$, CARA reduces to the expanded form of the algorithm found in \cite{bredies2024graph}, and when $\diag(Z) = 2\1$, it further reduces to the algorithm found in \cite{tam2023frugal} and \cite{bassett2024optimaldesignresolventsplitting} for a suitable rescaling of $\v{v}$, $\alpha$, and $\gamma$.

\section{Coupled Adaptable Backward-Forward-Backward Resolvent Splitting}
\label{cab:Sec:CABRA}
We now turn to CABRA, which uses two additional sets of matrices to include the differentiable terms.
As with the proximal operators, we structure the algorithm so that we evaluate each gradient operator once per iteration.
In this section, we describe their construction, the requirements they must satisfy for convergence, some possible parameter choices, and a proof of the convergence of CABRA.

For $k \in \mathcal{K}$, let $K_k$ and $Q_k$ be two sets of matrices which, respectively, govern the input and output for the gradient evluations.
Each $K_k$ forms a linear combination of the proximal function outputs associated with $y_k$ for input to the gradient evaluations using $y_k$, and each $Q_k$ linearly combines the outputs of those gradient calculations for use by subsequent proximal functions.
Letting $\mathcal{J}_k$ be the strictly increasing tuple of indices of the differentiable functions $g_j$ using element $y_k$, and $m_k = |\mathcal{J}_k|$ give the number of differentiable functions using element $y_k$, we select $Q_k \in \R^{n_k \times m_k}$ and $K_k \in \R^{m_k \times n_k}$.
When $m_k = 0$, we note that $Q_k$ and $K_k$ are linear maps to and from the zero-dimensional vector space, with $Q_k K_k = 0_{n \times n}$.
As with $Z_k$ and $W_k$, we address the element of $K_k$ in the row associated with $g_j$ and the column associated with $f_i$ as $(K_k)_{ji}$ rather than $(K_k)_{\pos(\mathcal{J}_k, j), \pos(\mathcal{I}_k, i) }$, with a similar convention applied to $Q_k$.
We use the notation $\kBj{j} \subseteq \mathcal{K}$ to give the strictly increasing tuple of the elements of $y$ selected by each operator $\RBj{j}$ for use in $g_j$, with $p_j' = |\mathcal{K}_j'|$.
We then create a function-centric lifted variable $\xb = (\xb_j)_{j \in \mathcal{J}}$, where each $\xb_j \in \R^{p_j'}$, and a variable-centric lifted variable $\yb = (\yb_k)_{k \in \mathcal{K}}$ where each $\yb_k \in \R^{m_k}$, addressing them as $\xb_{jk}$ and $\yb_{kj}$ as before.

We use the output of at least one proximal operator to build the input to each gradient, and provide the gradient output to at least one proximal operator, and therefore select the matrices so that the gradient evaluation is sequenced appropriately between those proximal evaluations.
We assume throughout that such a sequencing is possible, and note that, in practice, we can extend the domain of any function $f_i$ to include an additional variable with no modification in order to meet this assumption.
Let $i^*_j$ give the index of the first proximal operator receiving an output from $\nabla g_j$.
No inputs to the gradient can come from any $i \geq i^*_j$, and 
we therefore require
\begin{subequations}
    \label{cab:cutoff_assumptions}
\begin{align}
(K_k)_{ji} &= 0 \quad \forall j \in \mathcal{J}, \forall i \geq i^*_j, \forall k \in \kBj{j}\\
(Q_k)_{ij} &= 0 \quad \forall j \in \mathcal{J}, \forall i < i^*_j, \forall k \in \kBj{j}.
\end{align}
\end{subequations}

We form the input $\xb_j$ to each gradient operator $\nabla g_j$ by forming a weighted average over one or more proximal operator output values $\v{x}_i$, using matrices $K_k$ for $k \in \mathcal{K}_j'$ to form each element as
 \[\xb_{jk} = \sum_{i=1}^{i^*_j - 1} \left(K_k\right)_{ji} \v{x}_{ik}.\]  %
The condition $K_k \1_{n_k} = \1_{m_k}$ ensures that $K_k$ provides a full weighted average of input value to each $\nabla g_j$.
The $Q_k$ matrices then distribute the gradient output values to the remaining proximal operators in $\mathcal{I}_k$, spreading the full value of each output element, so that we require $\1^\intercal_{n_k} Q_k = \1_{m_k}^\intercal$.

Another way of viewing these constraints is to set up a directed graph $G_k$ over the $n_k+m_k$ nodes associated with the proximal and gradient computations for $y_k$. 
We create edges $(f_{i'},f_i)$ in this graph for any $i > i'$ where $(Z_k)_{ii'} \neq 0$.
These edges correspond with direct forward communications between proximal operators.
We also create edges $(f_i,g_j)$ for any $(K_k)_{ji} \neq 0$ and edges $(g_j,f_i)$ for any $(Q_k)_{ij} \neq 0$. 
The constraints above, summarized in \eqref{cab:cabra_reqs}, then require $G_k$ to be a connected directed acyclic graph beginning and ending on proximal computation nodes for all $k$.

The final set of requirements on $K_k$ and $Q_k$ preserve the nonexpansivity of the iterations.
This nonexpansivity is directly influenced by the cocoercivity constants $\beta_j$. 
Letting 
\begin{equation}
\label{cab:beta_k}
\beta_k = (\beta_j)_{j \in \mathcal{J}_k}
\end{equation} 
denote a vector of the cocoercivity constants associated with the gradients using $y_k$, we define 
\begin{equation}
    \label{cab:Uk}
U_k = (Q_k^\intercal - K_k)^\intercal \diag(\beta_k)^{-1} (Q_k^\intercal - K_k), 
\end{equation}
and then impose the requirement that $ l Z_k \succeq U_k$ for some scalar $l \geq 0 $.
The constant $l$ allows us to capture any ``slack'' in the requirement, and is particularly useful for capturing the relaxation of this constraint when there are no differentiable functions ($m=0$), in which case the $Z_k \succeq U_k$ constraint is trivially satisfied, and we can set $l=0$.
Setting $l=0$ also captures the case of $0$-smooth (affine) functions.
We therefore summarize the CABRA requirements as
\begin{subequations}
    \label{cab:cabra_reqs}
\begin{align}
    &Z_k \succeq W_k \label{cab:cabra_Z_succeq_W} \quad && \forall k \in \mathcal{K}\\
    &l Z_k \succeq U_k \label{cab:cabra_Z_succeq_U}\quad && \forall k \in \mathcal{K} \\
    &\nullspace(Z_k) = \text{span}(\1_{n_k})\label{cab:cabra_Z_sum}\quad && \forall k \in \mathcal{K}\\ 
    &\nullspace(W_k) = \text{span}(\1_{n_k})\label{cab:cabra_W_connect}\quad && \forall k \in \mathcal{K}\\
    &Q_k^\intercal \1_{n_k} = \1_{m_k} \label{cab:cabra_Q_sum}\quad && \forall k \in \mathcal{K}\\
    &K_k \1_{n_k} = \1_{m_k} \quad && \forall k \in \mathcal{K}\label{cab:cabra_K_sum}\\
&(K_k)_{ji} = 0 \quad &&\forall j \in \mathcal{J}, \forall i \geq i^*_j, \forall k \in \kBj{j}\\
&(Q_k)_{ij} = 0 \quad &&\forall j \in \mathcal{J}, \forall i < i^*_j, \forall k \in \kBj{j}.
\end{align} 
\end{subequations}
We then create block diagonal matrices
\begin{equation}
    \v{K} = \begin{pmatrix}
        K_1 &     & & \\
            & K_2 & & \\
            &     & \ddots & \\
            &     &       & K_p 
    \end{pmatrix} \quad \v{Q} = \begin{pmatrix}
        Q_1 &     & & \\
            & Q_2 & & \\
            &     & \ddots & \\
            &     &       & Q_p 
    \end{pmatrix}.
\end{equation}
As with CARA, we form permutations of $\v{K}$ and $\v{Q}$ which operate on a function-ordered space rather than a variable-ordered space using the permutation matrix \[\PB\colon \prod_{k=1}^p \R^{m_k} \to \prod_{j=1}^m \R^{p_j'}, \]
so that for $\xb = \PB \yb$, we have $\xb_{jk} = \yb_{kj}$.

We let $\KA = \PB \v{K} \PAs$ and $\QA = \PA \v{Q} \PBs$, and $\v{B} = (\nabla g_j)_{j \in \mathcal{J}}$.
Using this notation, we state CABRA succinctly in Algorithm \ref{cab:cabra_n_alg}.
\begin{algorithm} 
    \caption{CABRA (expanded form)}\label{cab:cabra_n_alg} 
    \begin{algorithmic}[1] 
    \Require $\alpha \in (0,\frac{4}{l})$; $\gamma_\nu \in (0, 2 - \frac{\alpha l}{2})$; $\v{v}^0$ satisfying \eqref{cab:v_req}; $\DA^{-1}$, $\WA$, $\LA$, $\QA$, $\KA$ satisfying \eqref{cab:cabra_reqs}.
    \State $\nu \gets 0$
    \Repeat
      \State $\v{x}^\nu = J_{\alpha \DA^{-1}\v{A}}\left(\DA^{-1}\left[\v{v}^\nu + 2\LA\v{x}^\nu - \alpha \QA\v{B}\KA\v{x}^\nu\right]\right)$ \label{cab:cabra_n_itr1}
      \State $\v{v}^{\nu+1} = \v{v}^\nu - \gamma_\nu \WA\v{x}^\nu$ \label{cab:cabra_n_itr2}
      \State $\nu \gets \nu+1$
    \Until{convergence}
\end{algorithmic}
\end{algorithm}

The next theorem establishes the convergence of CABRA. It follows the same structure outlined for the proof of CARA, with modification to account for the forward steps.
\begin{theorem}
    \label{cab:cabra_theorem}
Let $(\v{x}^\nu)$ and $(\v{v}^\nu)$ be the sequences of iterates produced by Algorithm \ref{cab:cabra_n_alg} applied to problem \eqref{cab:argmin_problem}, with $A_i = \partial f_i$ and $B_j = \nabla g_j$, for a given sequence of step sizes $\gamma_\nu \in (0, 2 - \frac{\alpha l}{2})$ with $\sum_\nu \gamma_\nu(2- \frac{\alpha l}{2}-\gamma_\nu) = \infty$. 
If the set of solutions of \eqref{cab:argmin_problem} is non-empty and there exists some $ x \in \R^p$ such that $ \RAi{i} x \in \text{ri}(\dom f_i)$ for all $i \in \mathcal{I}$ and $ \RBj{j} x \in \text{ri}(\dom g_j)$ for all $j \in \mathcal{J}$, then 
\begin{enumerate}
    \item $(\v{x}^\nu)$ converges to a point $\v{x}^*$ such that $\v{x}^*_i =\RAi{i} y^*$ where $y^*$ is a solution to problem \eqref{cab:argmin_problem}.
    \item $(\v{v}^\nu)$ converges to a point $\v{v}^*$ such that \[\v{v}_i^* = \alpha \v{w}_i^* + \alpha \sum_{j=1}^{m}\QA_{ij}\v{b}_j^* + (\ZA_{ii} + 2\sum_{i' < i} \ZA_{ii'})\v{x}^*_i\] for $\v{w}_i^* \in \partial f_i(\v{x}_i^*)$ and $\v{b}_j^* = \nabla g_j(\RBj{j} y^*)$. %
    \item $\begin{aligned}
    \sum_{k \in \mathcal{K}}\sum_{i \in \mathcal{I}_k}\norm{\v{x}_{ik}^\nu - \bar{y}_{k}^\nu}^2 = o\left(\frac{1}{\nu +1}\right)
    \end{aligned}$
    \item $\begin{aligned}
        \norm{\sum_{i \in \mathcal{I}}\RAi{i}^* \v{w}_i^\nu + \sum_{j \in \mathcal{J}} \left(\RBj{j}\right)^* \v{b}_j^\nu}^2 = o\left(\frac{1}{\nu +1}\right)
        \end{aligned}$
\end{enumerate}
where $\v{b}_j^\nu = \nabla g_j \left(\KA_{j}\v{x}^\nu\right)$ and $\v{w}^\nu = \frac{1}{\alpha}\left(\v{v}^\nu + (2\LA - \DA)\v{x}^\nu \right) - \QA \v{b}^\nu$.
\end{theorem}
\begin{proof}
    The addition of the cocoercive terms requires us to establish the conditions under which the iterates remain averaged-nonexpansive.
    Lemma \ref{cab:Lemma:PBK_coupled} relates $\QA\v{B}\KA$ and $\UA$, allowing us to show that the iterations of the algorithm remain averaged-nonexpansive within a (now) reduced set of step sizes given by $\gamma \in (0, 2 - \frac{\alpha l}{2})$ and scaling factors $\alpha \in (0, \frac{4}{l})$.
    Theorem \ref{cab:cabra_theorem} then follows directly from the proof of the more general case in the appendix in the same manner as Theorem \ref{cab:cara_theorem}.
\end{proof}

We note a few specific instances of problem \eqref{cab:argmin_problem} in order to illustrate the choice of $\v{Q}$ and $\v{K}$.
In the first case, suppose $\RAi{1} = \I$, $m= n-1$, and $\kA_{j+1} = \kBj{j}$, so that $g_1$ operates on the same space as $f_2$, $g_2$ and $f_3$ share a space, and so forth.
We could then choose $K_k$ to have a vector of ones in the first column and zeros elsewhere, and $(Q_k)_{ij} = 1$ for all $i = j+1$, with zeros elsewhere, so
\[ K_k = \begin{pmatrix}
    \1_{m_k } & \v{0}_{m_k \times (n_k-1)}
\end{pmatrix} \; \text{and} \; Q_k = \begin{pmatrix}
    \v{0}_{m_k}^\intercal \\
    \I_{n_k-1}
\end{pmatrix}.\]
This choice of matrices directs the output of the proximal operator of $f_1$ to each of the gradients, and then provides the gradient output from $\nabla g_j$ to the proximal operator of function $f_i$ where $i = j+1$.
In this case, if $\beta_k = \1$, $U_k$ is the graph Laplacian of the star graph in \eqref{cab:star}, and $Z_k \succeq U_k$ for $Z_k$ as the graph Laplacian of the star graph or the complete graph; otherwise we can scale up $Z_k$ to satisfy \eqref{cab:cabra_Z_succeq_U} or capture the required scaling in $l$.
If we choose the star graph in \eqref{cab:star} for each $Z_k$, CABRA takes the form
\begin{align*}
    \v{x}^{\nu}_1 &= \prox_{\alpha \DA_1^{-1} f_1}\left(\DA_1^{-1} \v{v}^{\nu}_1\right) &\\
    \v{x}^{\nu}_i &= \prox_{\alpha f_i}\left(\v{v}^{\nu}_i + 2 \RAi{i} \v{x}^{\nu}_1 - \alpha \nabla g_{i-1}(\RAi{i} \v{x}^{\nu}_1)\right) &\quad \text{in parallel for}\; i=2, \dots, n\\
    &\!\left. \begin{aligned} 
        \mathllap{\v{v}^{\nu+1}_1} &= \v{v}^{\nu}_1 - \gamma \left(\DA_1 \v{x}^{\nu}_1 - \sum_{i=2}^n \RAi{i}^* \v{x}^{\nu}_i\right) \\
        \mathllap{\v{v}^{\nu+1}_i} &= \v{v}^{\nu}_i - \gamma\left(\v{x}^{\nu}_i - \RAi{1} \v{x}^{\nu}_1\right)
    \end{aligned}\right\} &\quad \text{in parallel for}\; i=1, \dots, n
\end{align*}
where $\DA_1 = \diag((n_1 - 1, n_2 - 1, \dots, n_p - 1))$ as previously.

In the second case, we assume $\RAi{1} = \I$ as before, and use the same $K_k$ to provide the output of the first proximal operator to each gradient, but now assume that we want to spread the element $k$ results of these gradients across all of the proximal operators which use $y_k$. 
We can choose $(Q_k)_{ij} = \frac{1}{n_k-1}$ for all rows $i \geq 2$, for all columns $j$, with zeros in the top row, so 
\[Q_k = \frac{1}{n_k-1}\begin{pmatrix}
    \v{0}_{1 \times m_k} \\
    \v{1}_{(n_k-1) \times m_k}
\end{pmatrix}.\]
In this case, $U_k$ is a weighted version of a complete graph, with weights 1 for edges connecting to node 1, and weights $-\frac{1}{n_k-1}$ for edges connected to nodes other than the first.
The graph Laplacians of both the star graph \eqref{cab:star} and the complete graph \eqref{cab:complete} again satisfy $Z_k \succeq U_k$ when $\beta_k = \1$ and $m_k \leq n_k - 1$.
If we choose the star graph for each $Z_k$, CABRA takes the form
\begin{subequations}
    \label{cab:cabra_star_spread}
\begin{align}
    \v{x}^{\nu}_1 &= \prox_{\alpha \DA_1^{-1} f_1}\left(\DA_1^{-1} \v{v}^{\nu}_1\right) &\\
    \tilde{\v{b}}_j &= \DA_1^{-1} (\RBj{j})^* \nabla g_{j}(\RBj{j} \v{x}^{\nu}_1)\\
    \v{x}^{\nu}_i &= \prox_{\alpha f_i}\left(\v{v}^{\nu}_i + 2 \RAi{i}\v{x}^{\nu}_1 - \alpha \sum_{j=1}^m \RAi{i} \tilde{\v{b}}_j\right) &\quad \text{in parallel for}\; i=2, \dots, n\\
    &\!\left. \begin{aligned} 
        \mathllap{\v{v}^{\nu+1}_1} &= \v{v}^{\nu}_1 - \gamma \left(\DA_1 \v{x}^{\nu}_1 - \sum_{i=2}^n \RAi{i}^* \v{x}^{\nu}_i\right) \\
        \mathllap{\v{v}^{\nu+1}_i} &= \v{v}^{\nu}_i - \gamma\left(\v{x}^{\nu}_i - \RAi{i} \v{x}^{\nu}_1\right)
    \end{aligned}\right\} &\quad \text{in parallel for}\; i=1, \dots, n
\end{align}
\end{subequations}
If we have $\kBj{j} \subseteq \kA_n $ for all $j$, we could also use the average of the output values from the first $n-1$ proxes as the input for the gradients and send their output only to the last prox.
In this case, we would have \[K_k = \frac{1}{n_k-1}\begin{pmatrix}
\v{1}_{m_k \times (n_k -1)} & \v{0}_{m_k \times 1}
\end{pmatrix} \; \text{and} \; Q_k = \begin{pmatrix}
    \v{0}_{(n_k-1)\times m_k} \\
    \v{1}_{1 \times m_k}
\end{pmatrix}.\]
Using the complete graph for each $Z_k$, CABRA would be given as
\begin{align*}
    \v{x}^{\nu}_1 &= \prox_{\alpha \DA_1^{-1} f_1}\left(\DA_1^{-1} \v{v}^{\nu}_1\right) & \\
    \v{x}^{\nu}_i &= \prox_{\alpha \DA_i^{-1} f_i}\left(\DA_i^{-1}\left[ \v{v}^{\nu}_i + 2 \RAi{i} \sum_{i' < i}(\RAi{i'})^* \v{x}^{\nu}_{i'}\right]\right) &\text{for}\; i = 2, \dots, n-1\\
    \v{b}_j &= \nabla g_j(\DA_j^{-1}\RBj{j} \sum_{i<n}\RAi{i}^* \v{x}^{\nu}_i)&\text{in parallel for}\; j=1, \dots, m\\
    \v{x}^{\nu}_n &= \mathrlap{\prox_{\alpha \DA_n^{-1} f_n}\left(\DA_n^{-1}\left[ \v{v}^{\nu}_n + 2 \RAi{i} \sum_{i' < n}\RAi{i'}^* \v{x}^{\nu}_{i'} - \alpha  \RAi{n} \sum_{j=1}^m (\RBj{j})^* \v{b}_j\right]\right) } \\
    \v{v}^{\nu+1}_i &= \v{v}^{\nu}_i - \gamma\left(\DA_i \v{x}^{\nu}_i - \RAi{i} \sum_{i' \neq i} \RAi{i'}^* \v{x}^{\nu}_{i'} \right)  &\text{in parallel for}\; i=1, \dots, n
\end{align*}
with $\DA_j = \diag((n_k-1)_{k \in \kBj{j}})$.

\section{Parameter Adaptation}
\label{cab:Sec:Opt}

\subsection{SDP matrix selection}
While, in many cases, we can select matrix parameters directly based on a desired graph structure, in other contexts we may want to select matrix parameters for CABRA to accommodate a wide variety of additional goals, such as targeting specific values of $\DA_i$ for preconditioning or reduction of computational complexity, establishing a parallelized structure, optimizing various properties of the matrices like connectivity or sparsity, or minimizing any other convex function or convex set restriction on our matrix parameters, all while potentially restricting communication between specific resolvent or gradient operations to match a desired communication structure.
If we only need to match a given sparsity structure and target diagonal values, the Sinkhorn-Knopp Algorithm \cite{sinkhorn1967concerning} can provide a tractable method \cite{barkley2025decentralizedsensornetworklocalization}.
Otherwise, we can use the semi-definite programming formulation given in \eqref{cab:main_prob} to find $Z_k$, $W_k$, $K_k$, and $Q_k$ for each $k \in \mathcal{K}$.

\begin{subequations}\label{cab:main_prob}
    \begin{align}
    \minimize_{K_k, Q_k, W_k, Z_k} \quad &\phi_k(K_k, Q_k, W_k, Z_k) \label{cab:obj}\\
    \mathrm{subject\, to} \quad& Z_k \succeq W_k \label{cab:Z_succeq_W_con}\\
    & \begin{bmatrix} Z_k& Q_k-K_k^\intercal\\Q_k^\intercal - K_k&\diag(\beta_k)\end{bmatrix} \succeq 0 \label{cab:Z_succeq_U_con}\\
    & \lambda_1(W_k) + \lambda_2(W_k) \geq c \label{cab:fiedler_con}\\
    & Z_k \1_{n_k} = 0  \label{cab:Z_sum_con}\\
    & W_k \1_{n_k} = 0  \label{cab:W_sum_con}\\
    & Q_k^\intercal \1_{n_k} = \1_{m_k} \label{cab:Q_sum_con}\\
    & K_k\1_{n_k} = \1_{m_k} \label{cab:K_sum_con}\\
    & (Q_k)_{ij} = 0 \quad\quad \forall j \in \mathcal{J}_k, \quad i < i^*_j \label{cab:Q_tri_con}\\
    & (K_k)_{ji} = 0 \quad\quad \forall j \in \mathcal{J}_k, \quad i \geq i^*_j \label{cab:K_tri_con} \\
    & (K_k, Q_k, W_k, Z_k) \in \mathcal{C}_k \label{cab:con7} .%
    \end{align}
\end{subequations}
Here $\mathcal{C}_k$ is some convex subset of $\R^{m_k \times n_k} \times \R^{n_k \times m_k} \times \Sp^{n_k} \times \Sp^{n_k}$, $\phi_k\colon \mathcal{C}_k \to (-\infty, \infty]$ is any proper lower semicontinuous convex function, and $c \in \R_{++}$ is a positive algebraic connectivity parameter.

\begin{theorem}
   Given $n_k$, $m_k$, and $\beta_k$ (as defined in \eqref{cab:beta_k}) for all $k \in \mathcal{K}$, and $i^*_j$ for each $j \in \mathcal{J}$, the $p$ sets of matrix parameters satisfying \eqref{cab:main_prob} satisfy the requirements of \eqref{cab:cutoff_assumptions} and \eqref{cab:cabra_reqs} and therefore provide a valid set of matrix parameters for CABRA.
\end{theorem} 
\begin{proof}{Proof.}
    Constraints \eqref{cab:Z_succeq_W_con}, \eqref{cab:Z_sum_con}, \eqref{cab:Q_sum_con}, \eqref{cab:K_sum_con} directly satisfy assumptions \eqref{cab:cabra_Z_succeq_W}, \eqref{cab:cabra_Z_sum}, \eqref{cab:cabra_Q_sum}, and \eqref{cab:cabra_K_sum}.
    Assumption \eqref{cab:W_connect} is satisfied by the combination of \eqref{cab:fiedler_con} and \eqref{cab:W_sum_con}.
    Assumption \eqref{cab:cabra_Z_succeq_U} is met by the definition of $U_k$ in \eqref{cab:Uk} and constraint \eqref{cab:Z_succeq_U_con} by Schur complementarity.
    Therefore matrix parameters which are feasible in \eqref{cab:main_prob} satisfy assumptions \eqref{cab:cabra_reqs}.
    Constraints \eqref{cab:Q_tri_con} and \eqref{cab:K_tri_con} directly satisfy \eqref{cab:cutoff_assumptions}.
    We therefore have the $p$ required sets of matrix parameters, and can form valid lifted operators for CABRA.
\end{proof}

In practice, the algebraic connectivity of $W_k$ has a strong influence on the rate of convergence of similar resolvent splitting algorithms \cite{bassett2024optimaldesignresolventsplitting,bredies2024graph,aakerman2025splitting}. 
This research also shows that convergence is improved when the spectrum of $Z_k$ is close to that of $W_k$.
Minimizing $\lambda_{\mathrm{max}}(Z_k) - \lambda_2(W_k) - \lambda_1(W_k)$ provides an objective which simultaneously compresses the spectrum of $Z_k$ toward that of $W_k$ while maximizing the algebraic connectivity of $W_k$.
We use this objective to develop the matrices in the next section.

\subsection{Parallel Execution}

One benefit of \eqref{cab:main_prob} as a parameter selection framework is its flexibility in accommodating multi-block parallel execution, which in many cases requires different sparsity patterns for $Z$, $W$, and $U$.
Consider the case where $n=6$, $m=2$, each $\RAi{i}$ and $\RBj{j}$ is the identity, and we want to use four compute nodes in a distributed manner to calculate the operator values while maximizing parallelization.

We can do so by splitting the proximal operators into three blocks, $(1,2)$, $(3,4)$, and $(5,6)$. 
We require $Z_{ij} = 0$ within each block so that the two members of each block can run without requiring input from one another within the iteration.
We also require $W_{i,i'} = 0$ for all $i$ in block 1 and all $i'$ in block 3 so that block 1 can begin iteration $\nu+1$ while block 3 finishes iteration $\nu$.
For all $j \in \mathcal{J} = \{1,2\}$, we require $K_{ji} = 0$ for $i > 2$ and $Q_{ij} = 0$ for $i < 4$ so that both gradient operators receive input only from block 1 and provide output only to block 3, allowing them to run in parallel with block 2.
Assuming all computations require the same amount of time, in steady state we see block 1 and 3 executing in parallel (with block 1 an iteration ahead of block 3), followed by block 2 and the gradient operators executing in parallel. 
In both cases, we use all four computation nodes, ensuring efficient use of available resources, and require no central coordination.
With $\beta = \1$, we use problem \eqref{cab:main_prob} to find the following matrices which satisfy these conditions while minimizing the objective function described above.
\begin{subequations}
    \label{cab:parallel_mats}
\begin{align*}
    Z &= \frac{1}{3} \begin{bmatrix}
    \hspace{.5em} 4 \I_{2 \times 2} \hspace{.9em} & \hspace{.5em} -\v{1}_{2 \times 2} \hspace{.9em} & -\v{1}_{2 \times 2} \hspace{1.3em} \\[1.5em]
    -\v{1}_{2 \times 2} & 4 \I_{2 \times 2} & -\v{1}_{2 \times 2} \\[1.5em]
    -\v{1}_{2 \times 2} & -\v{1}_{2 \times 2} & 4 \I_{2 \times 2}
\end{bmatrix}, &
    W &= \frac{1}{9} \begin{bmatrix}
    \begin{bmatrix} 8 & -2 \\ -2 & 8 \end{bmatrix} & -3\v{1}_{2 \times 2} & \v{0}_{2 \times 2} \\
    -3\v{1}_{2 \times 2} & \begin{bmatrix} 11 & 1 \\ 1 & 11 \end{bmatrix} & -3\v{1}_{2 \times 2} \\
    \v{0}_{2 \times 2} & -3\v{1}_{2 \times 2} & \begin{bmatrix} 8 & -2 \\ -2 & 8 \end{bmatrix}
\end{bmatrix} \\
    U &= \frac{1}{2} \begin{bmatrix}
    \hspace{1em} \v{1}_{2 \times 2} \hspace{1em} & \hspace{1em}\v{0}_{2 \times 2} \hspace{1em} & -\v{1}_{2 \times 2} \hspace{2em} \\[1.5em]
    \v{0}_{2 \times 2} & \v{0}_{2 \times 2} & \v{0}_{2 \times 2} \\[1.5em]
    -\v{1}_{2 \times 2} & \v{0}_{2 \times 2} & \v{1}_{2 \times 2}
\end{bmatrix}, & \quad K &= \frac{1}{2} \begin{bmatrix}
    \v{1}_{2 \times 2} \\[1.5em]
    \v{0}_{2 \times 2} \\[1.5em]
    \v{0}_{2 \times 2}
\end{bmatrix}^\intercal , \quad Q = \frac{1}{2} \begin{bmatrix}
    \v{0}_{2 \times 2} \\[1.5em]
    \v{0}_{2 \times 2} \\[1.5em]
    \v{1}_{2 \times 2}
\end{bmatrix}.
\end{align*}
\end{subequations}

\section{Numerical Examples}
\label{cab:Sec:Experiments}

\subsection{Diagonal Scaling}
We now provide a small example of the potential use of these matrix parameters for diagonally scaling the operators.
The code to reproduce the results in this paper, along with additional experiments demonstrating the value of preconditioning in other problems, can be found online at \url{https://www.github.com/peterbarkley/cabra_paper}.

We consider the problem of finding a point in the intersection of several halfspaces, beginning with those given by
\begin{align*}
0.05x_1 - x_2 &\geq 2 \;\text{ and} \\
0.05x_1 + x_2 &\geq 2.
\end{align*} 
We formulate this as the instance of problem \eqref{cab:argmin_problem} with $n=2$, $m=0$, and $p=2$ given by
\begin{equation*}
\text{find} \;y \in \argmin_{x\in\R^2} \iota_1(x) + \iota_2(x)
\end{equation*}
where $\iota_1$ and $\iota_2$ are the indicator functions on the halfspaces.
Choosing \[Z_1 = Z_2 = W_1 = W_2 = \begin{bmatrix}
    1 & -1 \\ -1 & 1
\end{bmatrix},\] initial values $\v{v}_1^0 = \v{v}_2^0 = (0,0)$, and $\gamma = 1.5$ for Algorithm \ref{cab:cara_n_alg}, we get the sequence of iterates $\v{x}_1^\nu$ and $\v{x}_2^\nu$ given by Figure \ref{cab:fig:precond_vanilla}.
If we instead scale $Z_1$ and $W_1$ by the square of the coefficient for $x_1$ in the halfspace constraints, we get \[\DA_1 = \DA_2 = \begin{bmatrix}
    0.0025 & 0 \\ 0 & 1
\end{bmatrix}.\]
Using these scaled matrix parameters and the previous values for $\gamma$ and $\v{v}^0$, we get the sequence of iterates $\v{x}_1^\nu$ and $\v{x}_2^\nu$ given by Figure \ref{cab:fig:precond_scaled}, where we converge in two steps rather than 16.
\begin{figure}[H]
    \centering
    \subfloat[Identical Matrices]{
    \centering        
    \includegraphics[width=0.45\textwidth]{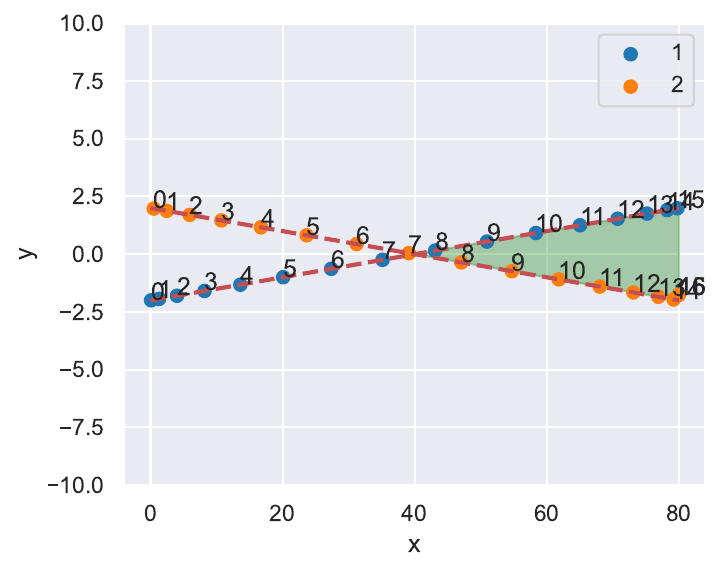}
    \label{cab:fig:precond_vanilla}
}
\hfil
\subfloat[Scaled Matrices]{
    \centering        
    \includegraphics[width=0.45\textwidth]{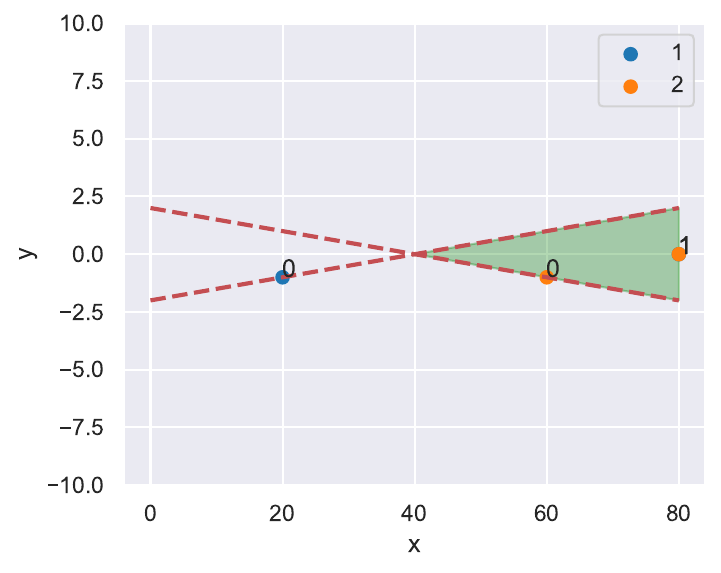}
    \label{cab:fig:precond_scaled}
}
\caption[Preconditioning demonstration]{Preconditioning demonstration. The iterate number is displayed above the points, with the results of the first projection in blue and those of the second projection in orange.
Note that, despite $\v{x}_1^8$ and $\v{x}_2^8$ being in the valid region in Figure \ref{cab:fig:precond_vanilla}, the algorithm does not converge until iteration 16, because the subgradient portions of $\v{v}_1$ and $\v{v}_2$ have accumulated non-zero values which generate a momentum in the iterates, which must then dissipate.
}
\label{cab:fig:precon}
\end{figure}
If we make the halfspace boundaries even more parallel, we can make the convergence of the unscaled matrices arbitrarily slow, while as long as the scaling is sufficiently high, we can achieve two-step convergence with the scaled matrices.
This scaling in the two-operator case can be shown to be an instance of the Douglas-Rachford scaling presented in \cite{giselsson2014preconditioning}.

Using CARA, however, we can apply the same approach to scaling to an arbitrary number of operators.
We illustrate this in Figure \ref{cab:fig:large_op}, which shows the convergence of the CARA algorithm under unscaled (uniform) and diagonally scaled preconditioning to find a point in the intersection of $n=30$ halfspaces in $\R^{200}$ which are chosen to be nearly parallel to one another. 
We select halfspaces which are close to parallel by selecting a random vector $c$ uniformly in $[0.5,2.5]^{200}$ and then form the first 15 halfspace normal vectors by adding $c$ to 15 random vectors selected uniformly in $[0,.1]^{200}$. 
We then form the last 15 normal vectors by subtracting $c$ from 15 random vectors selected uniformly in $[0,.1]^{200}$. 
We then form the halfspaces $c_i^\intercal x_i \geq v_i$, with the constants $v_i$ are select at random uniformly in $[0,1]$.
We choose the uniform weight matrix parameters as given by \eqref{cab:complete} scaled by $\frac{1}{n_k-1}$ to have unit diagonal, so that
\begin{equation}
    \label{cab:uniform_mats}
Z_k = W_k = \frac{n_k}{n_k-1} \I_{n_k} - \frac{1}{n_k-1}\1_{n_k \times n_k}.
\end{equation}
We form scaled matrices by multiplying each $Z_k$ and $W_k$ by $c_{k}^2$.
We see that the scaled matrices converge more quickly to a point in the intersection of the halfspaces.

\begin{figure}[H]
    \centering
    \subfloat[Identical Matrices]{
    \centering        
    \includegraphics[width=0.45\textwidth]{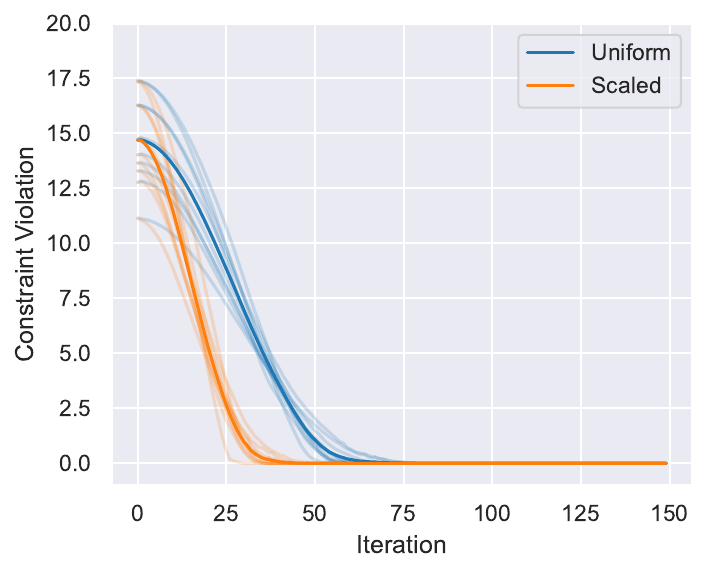}
    \label{cab:fig:large_op_constraint_violation}
}
\hfil
\subfloat[Scaled Matrices]{
    \centering        
    \includegraphics[width=0.45\textwidth]{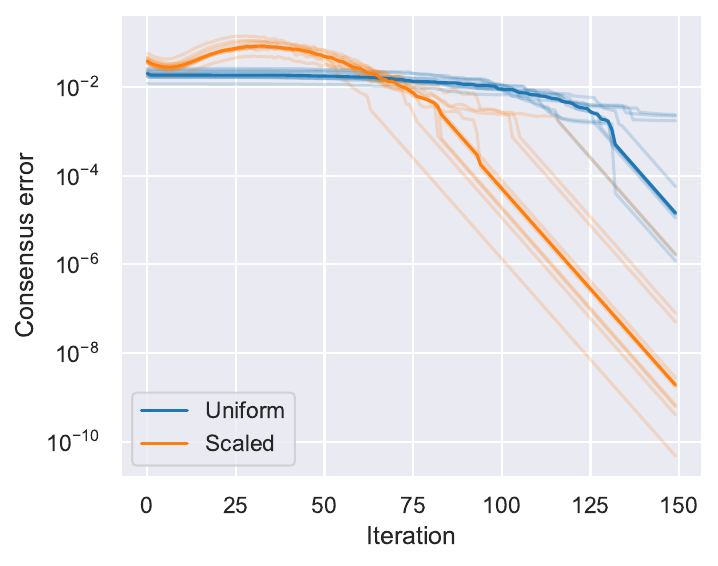}
    \label{cab:fig:large_op_consensus_error}
}
\caption[Preconditioning 30 operators]{We extend the preconditioning approach shown in Figure \ref{cab:fig:precon} to a thirty-operator problem.
}
\label{cab:fig:large_op}
\end{figure}

\subsection{Stochastic Weapon Target Allocation}
We now consider the application of CABRA to the continuous relaxation of a multi-stage stochastic weapon target assignment problem \cite{manne1958target,li2016wta} with a finite set of scenarios $\mathcal{S}$, stages $\mathcal{T}$, targets $\mathcal{E}$, and weapons $\mathcal{W}$.
In this problem, we make an initial weapon allocation in stage 1, after which the updated probability of success of weapon $i$ against target $j$ becomes known for each branch in the stage, and we have an opportunity to make a stage 2 weapon allocation, with this process repeating in each stage.
In all scenarios, we require the sum of the allocated fractions of weapon $i$ used over the stages to be less than or equal to one.

The set $\mathcal{X}_t$ defines the branches of the scenario tree at stage $t$, and the subset $\mathcal{S}_{bt} \subseteq \mathcal{S}$ defines the scenarios in branch $b$ over which we enforce equality for stage $t$.
We require the following problem parameters and variables:
\begin{align*}
    w_s &\in (0,1) \quad &&\forall s \in \mathcal{S}&& \text{probability of scenario}\, s\\
    V_j &\in (0,1) \quad &&\forall j \in \mathcal{E}&& \text{value of target}\, j\\
    b_{st} &\in \mathcal{X}_t \quad &&\forall s,t \in \mathcal{S} \times \mathcal{T}&& \text{branch}\, b \,\text{such that}\, s \in \mathcal{S}_{bt} \\
    p_{tsij} &\in (0,1) \quad &&\forall i,j,s,t \in \mathcal{W} \times \mathcal{E} \times \mathcal{S} \times \mathcal{T} && \text{damage parameter for target}\, j \,\text{by weapon}\, i \\
    &&&&& \text{in scenario}\, s \,\text{and stage} \, t\\
    x_{tbij} &\in \R \quad &&\forall (i,j,t) \in \mathcal{W} \times \mathcal{E} \times \mathcal{T},\, b \in \mathcal{X}_t & &\text{amount of weapon}\, i \, \text{to use in stage}\, t \\
    &&&&& \text{against target}\, j \,\text{ in all scenarios}\, s \in \mathcal{S}_{bt}  
\end{align*}
so that \[x \in \prod_{t \in \mathcal{T}}\prod_{b \in \mathcal{X}_t}\R^{|\mathcal{W}| \times |\mathcal{E}|}.\]
We also define $q_{tsij} = -\log(1 - p_{tsij})$.
Using this notation, we define the continuous relaxation of the multi-stage stochastic weapon target assignment problem as
\begin{subequations}
    \label{cab:rswta}
\begin{align}
    \minimize_{x } \quad & \sum_{s \in \mathcal{S}} w_s \sum_{j \in \mathcal{E}} V_j \exp\left(-\sum_{i \in \mathcal{W}} \sum_{t \in \mathcal{T}}q_{tsij}x_{t, b_{st},i,j}\right) \\
        \text{s.t.} \quad & x_{tbij} \geq 0 \quad\forall i,j,t,b \in \mathcal{D}\\
        & \sum_{t \in \mathcal{T}}\sum_{j \in \mathcal{E}} x_{t, b_{st},i,j} \leq 1 \quad \forall i,s \in \mathcal{W} \times \mathcal{S} .
\end{align}
\end{subequations}
We then define 
\begin{subequations}
\begin{align}
    R_{s}(x) &= (x_{t, b_{st},i,j})_{i \in \mathcal{W}, j \in \mathcal{E}, t \in \mathcal{T}} \quad &\forall s \in \mathcal{S} \\
    \iota_0 (x) &= \begin{cases}
        0 \quad & \text{if} \; x \geq 0 \\
        \infty \quad & \text{otherwise}
    \end{cases}&\\
    \iota_1 (R_{s}x) &= \begin{cases}
        0 \quad & \text{if} \; \sum_{j \in \mathcal{E}}\sum_{t \in \mathcal{T}} x_{t, b_{st},i,j} \leq 1 \quad \forall i \in \mathcal{W}\\
        \infty \quad & \text{otherwise}
    \end{cases}\quad&\forall s \in \mathcal{S} \\
    g_{s}(R_{s} x) &= w_s \sum_{j \in \mathcal{E}} V_j \exp(-\sum_{i \in \mathcal{W}}\sum_{t \in \mathcal{T}}q_{tsij}x_{t, b_{st},i,j})  \quad& \forall s \in \mathcal{S} .
\end{align}
\end{subequations}
Problem \eqref{cab:rswta} can then be rewritten as 
\begin{equation}
    \label{cab:wta_min_ind}
    \min_x \, \iota_0(x) + \sum_{s \in \mathcal{S}}\iota_1(R_{s}x) + \tau\sum_{s \in \mathcal{S}} g_{s}(R_{s} x)
\end{equation}
for any scaling $\tau > 0$.
We note that the interior of the feasible region is non-empty, and problem \eqref{cab:rswta} meets the regularity conditions required for equivalence with problem \eqref{cab:zero_in_coupled_sum}.

We let $x_{t b} \in \R^{|\mathcal{W}| \times |\mathcal{E}|}$ be the submatrix of $x$ corresponding to stage $t$ and branch $b$.
Note that each matrix $x_{t b}$ is an argument to each $\iota_1$ function associated with a scenario in $\mathcal{S}_{bt}$, and to $\iota_0$, so that for every variable index $k$ in $x_{tb}$, $n_k = |\mathcal{S}_{bt}| + 1$. 
We use $n_{tb}$ to represent $n_k$ for every variable in $x_{tb}$, since the value is constant.
Each $x_{t b}$ is also an argument in $|\mathcal{S}_{bt}|$ derivative operators $\nabla g_{s}$, so that for every variable index $k$ in $x_{tb}$, $m_k = |\mathcal{S}_{bt}| = n_{tb} -1 $. 
We also note that the proximal operators of $\iota_0$ and each $\iota_1(R_s)$ are fully separable across the weapons.

Since the output of the proximal operator of $\iota_0$ is non-negative, we structure the splitting to provide that output value as the argument to each $\nabla g_{s}$.
On this non-negative domain, $\nabla g_{s}$ is $(w_s \max_{j \in \mathcal{E}}V_j \norm{q_{js}}^2)^{-1}$-cocoercive where $q_{js} = (q_{tsij})_{i,t \in \mathcal{W} \times \mathcal{T}}$.
We set $\tau = \left(w_s \max_{js} (V_j \norm{q_{js}}^2)\right)^{-1}$, which rescales the problem so that $\beta_{s} \geq 1$ for all derivatives.
We can then formulate a parallelizable splitting over the $|W|$ weapons platforms with the following matrix parameters, based on the star graph approach in \eqref{cab:cabra_star_spread} for each element of $x_{t b}$:
\begin{subequations}
    \label{cab:matrices}
\begin{align}
    Z_{ tb} = W_{ tb} &\in \Sp^{n_{tb}} & Z_{tb} = W_{tb} &= \begin{bmatrix}
        n_{tb}-1 & -\v{1}_{1 \times (n_{tb}-1)} \\
        -\v{1}_{(n_{tb}-1) \times 1}     & \I_{n_{tb}-1}
    \end{bmatrix}\\
    K_{ tb} & \in \R^{(n_{tb}-1) \times n_{tb}} & K_{tb} &= \begin{bmatrix}
        \v{1}_{(n_{tb}-1) \times 1} & \v{0}_{(n_{tb}-1) \times (n_{tb}-1)}
    \end{bmatrix}\\
    Q_{ tb} & \in \R^{n_{tb} \times (n_{tb}-1)} & Q_{tb} &= \frac{1}{n_{tb}-1}\begin{bmatrix}
        \v{0}_{1 \times (n_{tb}-1)} \\ \v{1}_{(n_{tb}-1) \times (n_{tb}-1)}
    \end{bmatrix} .%
\end{align}
\end{subequations}
These matrices satisfy \eqref{cab:cabra_reqs}.

Figure \ref{fig:wta_splitting} depicts the communication structure implied by these matrix selections when we have four scenarios with a branching factor of two in both the first and second stage.
Each arrow depicts the output of a proximal operator (or gradient) being used as part of the argument for a subsequent operator.
The star graph nature of each $Z_k$ is evident in the linkage from the blocks of the first proximal operator to each of the blocks of the halfspace proximal operators, with no linkages between the different halfspace operators. 
The $K_k$ matrices also exhibit this star graph structure, evident in the linkages from the first row to the second row, and the $Q_k$ matrix structure defines the many-to-many linkages which provide a scaled sum of the output of the gradients to each corresponding halfspace projection block

\begin{figure}[H]
    \centering
    \resizebox{\textwidth}{!}{%
    \begin{tikzpicture}[
        mat5x6/.style={
            rectangle,
            draw=black!80,
            fill=white,
            thick,
            minimum width=0.9cm,  %
            minimum height=0.75cm, %
            inner sep=0pt,
            path picture={
                \draw[step=0.15cm, gray!30, thin] (path picture bounding box.south west) grid (path picture bounding box.north east);
            },
            align=center,
            font=\footnotesize\bfseries
        },
        operator/.style={
            draw,
            thick,
            rounded corners,
            inner sep=0.4cm
        },
        stage1 arrow/.style={-latex, thick, red!70!black, opacity=0.7},
        stage2 arrow/.style={-latex, thick, blue!70!black, opacity=0.7},
        stage3 arrow/.style={-latex, thick, green!50!black, opacity=0.7},
        stage1 long arrow/.style={-latex, thick, red!70!black, opacity=0.25},
        stage2 long arrow/.style={-latex, thick, blue!70!black, opacity=0.25},
        stage3 long arrow/.style={-latex, thick, green!50!black, opacity=0.25}
    ]

    \begin{scope}[local bounding box=L1]
        \node[mat5x6] (L1_11) at (-6, 0) {$x_{1,1}$};
        
        \node[mat5x6] (L1_21) at (-2, 0) {$x_{2,1}$};
        \node[mat5x6] (L1_22) at (-0.5, 0) {$x_{2,2}$};
        
        \node[mat5x6] (L1_31) at (3, 0) {$x_{3,1}$};
        \node[mat5x6] (L1_32) at (4.5, 0) {$x_{3,2}$};
        \node[mat5x6] (L1_33) at (6, 0) {$x_{3,3}$};
        \node[mat5x6] (L1_34) at (7.5, 0) {$x_{3,4}$};
    \end{scope}
    \begin{scope}[on background layer]
        
    \node[operator, fill=yellow!5, dashed, fit=(L1_11) (L1_34), label=above:{Projection onto Non-negative Halfspace ($\text{prox}_{\iota_0}$)}] (L1_box) {};
    \end{scope}

    \begin{scope}[yshift=-5cm]
        \node[mat5x6] (L2_s1_1) at (-7.7, 0) {$x_{1,1}$};
        \node[mat5x6] (L2_s1_2) at (-6.5, 0) {$x_{2,1}$};
        \node[mat5x6] (L2_s1_3) at (-5.3, 0) {$x_{3,1}$};
        
        \node[mat5x6] (L2_s2_1) at (-3.2, 0) {$x_{1,1}$};
        \node[mat5x6] (L2_s2_2) at (-2.0, 0) {$x_{2,1}$};
        \node[mat5x6] (L2_s2_3) at (-0.8, 0) {$x_{3,2}$};
        
        \node[mat5x6] (L2_s3_1) at (1.3, 0) {$x_{1,1}$};
        \node[mat5x6] (L2_s3_2) at (2.5, 0) {$x_{2,2}$};
        \node[mat5x6] (L2_s3_3) at (3.7, 0) {$x_{3,3}$};
        
        \node[mat5x6] (L2_s4_1) at (5.8, 0) {$x_{1,1}$};
        \node[mat5x6] (L2_s4_2) at (7.0, 0) {$x_{2,2}$};
        \node[mat5x6] (L2_s4_3) at (8.2, 0) {$x_{3,4}$};
    \end{scope}

    \begin{scope}[yshift=-10cm]
        \node[mat5x6] (L3_s1_1) at (-7.7, 0) {$x_{1,1}$};
        \node[mat5x6] (L3_s1_2) at (-6.5, 0) {$x_{2,1}$};
        \node[mat5x6] (L3_s1_3) at (-5.3, 0) {$x_{3,1}$};
        
        \node[mat5x6] (L3_s2_1) at (-3.2, 0) {$x_{1,1}$};
        \node[mat5x6] (L3_s2_2) at (-2.0, 0) {$x_{2,1}$};
        \node[mat5x6] (L3_s2_3) at (-0.8, 0) {$x_{3,2}$};
        
        \node[mat5x6] (L3_s3_1) at (1.3, 0) {$x_{1,1}$};
        \node[mat5x6] (L3_s3_2) at (2.5, 0) {$x_{2,2}$};
        \node[mat5x6] (L3_s3_3) at (3.7, 0) {$x_{3,3}$};
        
        \node[mat5x6] (L3_s4_1) at (5.8, 0) {$x_{1,1}$};
        \node[mat5x6] (L3_s4_2) at (7.0, 0) {$x_{2,2}$};
        \node[mat5x6] (L3_s4_3) at (8.2, 0) {$x_{3,4}$};
    \end{scope}

    \begin{scope}[on background layer]
        \node[operator, fill=blue!5, fit=(L2_s1_1) (L2_s1_3), label={[xshift=-1.2cm]above:{$\nabla f_1$}}] (L2_op1) {};
        \node[operator, fill=blue!5, fit=(L2_s2_1) (L2_s2_3), label={[xshift=-1.2cm]above:{$\nabla f_2$}}] (L2_op2) {};
        \node[operator, fill=blue!5, fit=(L2_s3_1) (L2_s3_3), label={[xshift=-1.2cm]above:{$\nabla f_3$}}] (L2_op3) {};
        \node[operator, fill=blue!5, fit=(L2_s4_1) (L2_s4_3), label={[xshift=-1.2cm]above:{$\nabla f_4$}}] (L2_op4) {};

        \node[operator, fill=green!5, fit=(L3_s1_1) (L3_s1_3), label=below:{Halfspace 1 ($\text{prox}_{\iota_1}$)}] (L3_op1) {};
        \node[operator, fill=green!5, fit=(L3_s2_1) (L3_s2_3), label=below:{Halfspace 2 ($\text{prox}_{\iota_1}$)}] (L3_op2) {};
        \node[operator, fill=green!5, fit=(L3_s3_1) (L3_s3_3), label=below:{Halfspace 3 ($\text{prox}_{\iota_1}$)}] (L3_op3) {};
        \node[operator, fill=green!5, fit=(L3_s4_1) (L3_s4_3), label=below:{Halfspace 4 ($\text{prox}_{\iota_1}$)}] (L3_op4) {};

        \foreach \j in {1,2,3,4} {
            \draw[stage1 long arrow] (L1_11.south) to[out=270, in=90] (L3_s\j_1.north);
        }

        \foreach \j in {1,2} {
            \draw[stage2 long arrow] (L1_21.south) to[out=270, in=90] (L3_s\j_2.north);
        }

        \foreach \j in {3,4} {
            \draw[stage2 long arrow] (L1_22.south) to[out=270, in=90] (L3_s\j_2.north);
        }

        \foreach \j in {1,2,3,4} {
            \draw[stage3 long arrow] (L1_3\j.south) to[out=270, in=90] (L3_s\j_3.north);
        }
    \end{scope}

    \draw[stage1 arrow] (L1_11.south) to[out=270, in=90] (L2_s1_1.north);
    \draw[stage1 arrow] (L1_11.south) to[out=270, in=90] (L2_s2_1.north);
    \draw[stage1 arrow] (L1_11.south) to[out=270, in=90] (L2_s3_1.north);
    \draw[stage1 arrow] (L1_11.south) to[out=270, in=90] (L2_s4_1.north);

    \draw[stage2 arrow] (L1_21.south) to[out=270, in=90] (L2_s1_2.north);
    \draw[stage2 arrow] (L1_21.south) to[out=270, in=90] (L2_s2_2.north);
    \draw[stage2 arrow] (L1_22.south) to[out=270, in=90] (L2_s3_2.north);
    \draw[stage2 arrow] (L1_22.south) to[out=270, in=90] (L2_s4_2.north);

    \draw[stage3 arrow] (L1_31.south) to[out=270, in=90] (L2_s1_3.north);
    \draw[stage3 arrow] (L1_32.south) to[out=270, in=90] (L2_s2_3.north);
    \draw[stage3 arrow] (L1_33.south) to[out=270, in=90] (L2_s3_3.north);
    \draw[stage3 arrow] (L1_34.south) to[out=270, in=90] (L2_s4_3.north);

    \foreach \i in {1,2,3,4} {
        \foreach \j in {1,2,3,4} {
            \draw[stage1 arrow] (L2_s\i_1.south) to[out=270, in=90] (L3_s\j_1.north);
        }
    }

    \foreach \i in {1,2} {
        \foreach \j in {1,2} {
            \draw[stage2 arrow] (L2_s\i_2.south) to[out=270, in=90] (L3_s\j_2.north);
        }
    }

    \foreach \i in {3,4} {
        \foreach \j in {3,4} {
            \draw[stage2 arrow] (L2_s\i_2.south) to[out=270, in=90] (L3_s\j_2.north);
        }
    }

    \foreach \i in {1,2,3,4} {
        \draw[stage3 arrow] (L2_s\i_3.south) -- (L3_s\i_3.north);
    }

    \end{tikzpicture}%
    }
    \caption{Visual representation of variable distribution in step \eqref{cab:cabra_n_itr1} the CABRA stochastic weapon target assignment splitting algorithm. This shows how the splitting would operate with a branching factor of 2, in which we have 2 nodes in the second stage and 4 nodes in the third stage. 
    The red lines represent operator output associated with the first stage ($x_{1,1}$) block of variables, the blue lines represent operator output associated with the second stage ($x_{2,1}$ and $x_{2,2}$) block of variables, and the green lines represent operator output associated with the third stage ($x_{3,i}$) of variables. The top row represents the projection onto the non-negative halfspace, the middle row represents the gradient evaluations for the different scenarios, and the bottom row represents the $\iota_1$ halfspace projections. }
    \label{fig:wta_splitting}
\end{figure}

The separability of the proximal operators over the weapons platforms allows us to execute the entire process in parallel across weapons platforms which sequentially compute the projection onto the non-negative halfspace and the weapon-specific sum required for each gradient, transmit that information to the other platforms, and then compute the weapon-specific components of the gradients and the projection onto the $\iota_1$ halfspaces for each scenario, concluding each iteration with an update of the components of $\v{v}$ specific to each weapon platform.
Figure \ref{fig:weapon_split} depicts this approach.
\begin{figure}
    \centering
    \includegraphics[width=\linewidth]{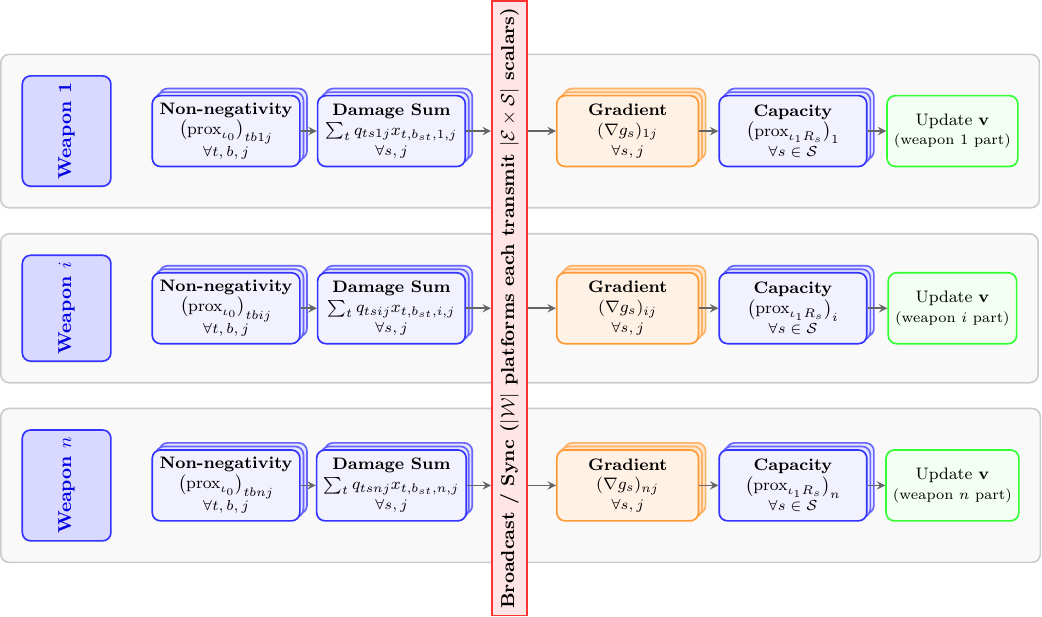}
    \caption{Depiction of a further splitting over the weapon platforms which exploits the separability of the resolvents over the weapons.}
    \label{fig:weapon_split}
\end{figure}
In this formulation, only one communication is required between platforms in each iteration, and that communication is reduced to only $|\mathcal{E} \times \mathcal{S}|$ scalars from each weapon platform, and can be even smaller when some platforms have $q_{tsij} = 0$ for some $j$ and $s$.

We test the CABRA splitting of problem \eqref{cab:wta_min_ind} by comparing it with the splitting given by PG-EXTRA \cite{pgextra} and the Graph-Based Forward Backward Splitting (GFBS) \cite{aragon2024forward} on $\partial \iota_0$, the composite maximal monotone operators $R_s^* \partial \iota_1 R_s$ and cocoercive operators $\tau R_s^* \nabla g_s R_s$, using a zero operator for the first differentiable operator in PG-EXTRA, since it requires $n=m$.
We use the star graph in \eqref{cab:star} for GFBS and a uniform mixing matrix for PG-EXTRA.
We generate problem test data for 10 weapons platforms, 20 targets, and a scenario tree over three stages and 100 scenarios, with $V_j = 1$,  $w_s = 0.01$, $\mathcal{X}_1 = \{ 1 \}$, $\mathcal{X}_2 = \{ 1, \dots, 10 \}$, and $\mathcal{X}_3 = \{ 1, \dots, 100 \}$.
Each test problem has a base damage parameter $\tilde{P} \in \R^{10 \times 20}$ selected uniformly from $[ .2, .8]^{10 \times 20}$.
In each stage, we model increasing knowledge of the environment and improved positioning in it by adding 0.01 to $\tilde{P}$, and model the variability leading to each branch as a random walk in which we move from each $P_{t,s}$ by adding a uniform random variable.
We do so by letting
\[ P_{t,s} = \tilde{P} + 0.01(t-1) \1_{10 \times 20} + \sum_{t' = 1}^{t} Y_{t'+1, b_{s,t'+1}}\]
where $Y_{tb} \in \R^{10 \times 20}$ is selected uniformly from $[ -0.05, .05]^{10 \times 20}$ as the perturbation leading to branch $b$ in stage $t$, letting $b_{4s} = b_{3s}$.

We conduct ten trials, beginning with $\gamma = .999$ and $\alpha=2$ for CABRA and GBFB, and $\alpha=1$ for PG-EXTRA (within the proven range of convergence).
We also test at $\alpha = 10$ for CABRA and GBFB, and $\alpha=2$ for PG-EXTRA, having manually evaluated $\alpha$ parameters and determined that the larger scaling parameter speeds convergence despite being beyond the proven set of valid parameters, which is reasonable since the exponential functions have monotonically decreasing gradients as the iterates move from the origin towards solutions satisfying the summation conditions in the first orthant, and thus the Lipschitz constants decrease as the algorithm progresses.

Figure \ref{cab:fig:swta} displays the results, providing the objective function gap and constraint violation for the mean of the proximal outputs, and the summation of the mean squared consensus error over the variable blocks. 
The bold lines display the median value over the tests for each algorithm, with the background lines depicting each test result.
The top row shows the results using parameters within the range of proven convergence.
We see that CABRA outperforms the other algorithms on every measure.
\begin{figure}[H]
    \centering
    \begin{subfigure}[b]{0.31\textwidth}
        \centering
        \includegraphics[width=.95\textwidth]{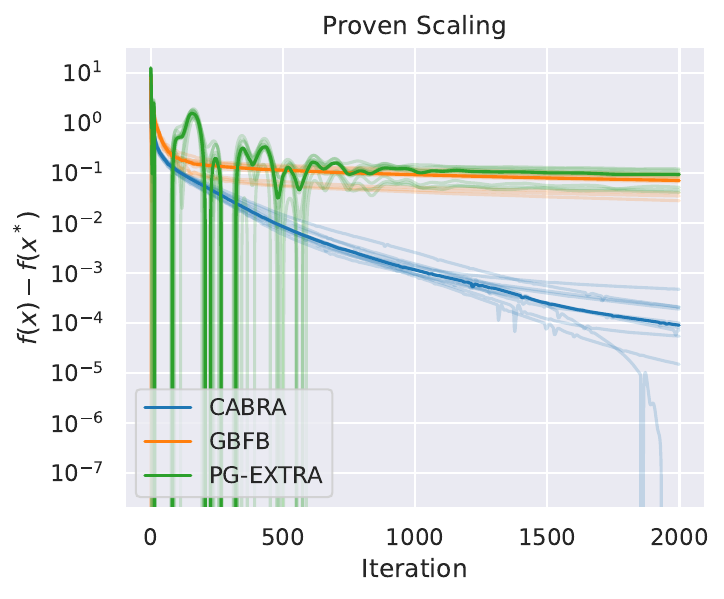} %
        \vspace{0.5em} %
        \includegraphics[width=.95\textwidth]{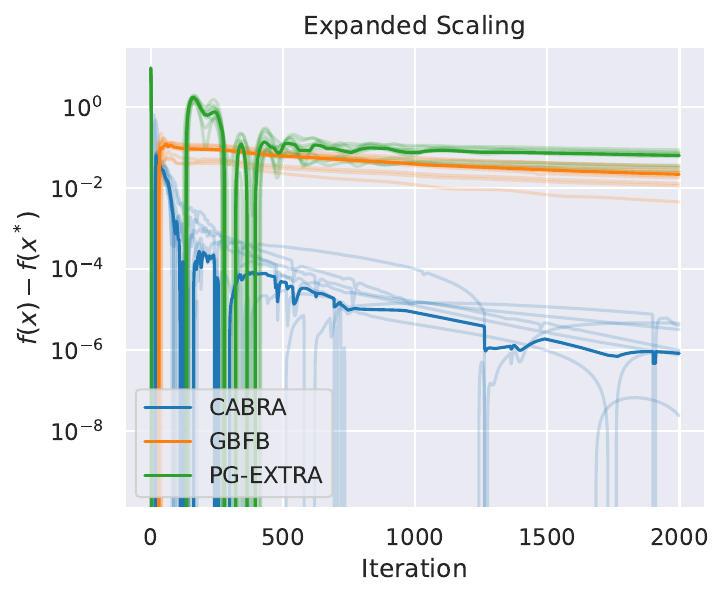} %
        \caption{Objective function gap}
        \label{cab:fig:swta_gap}
    \end{subfigure}
    \hfill
    \begin{subfigure}[b]{0.31\textwidth}
        \centering
        \includegraphics[width=\textwidth]{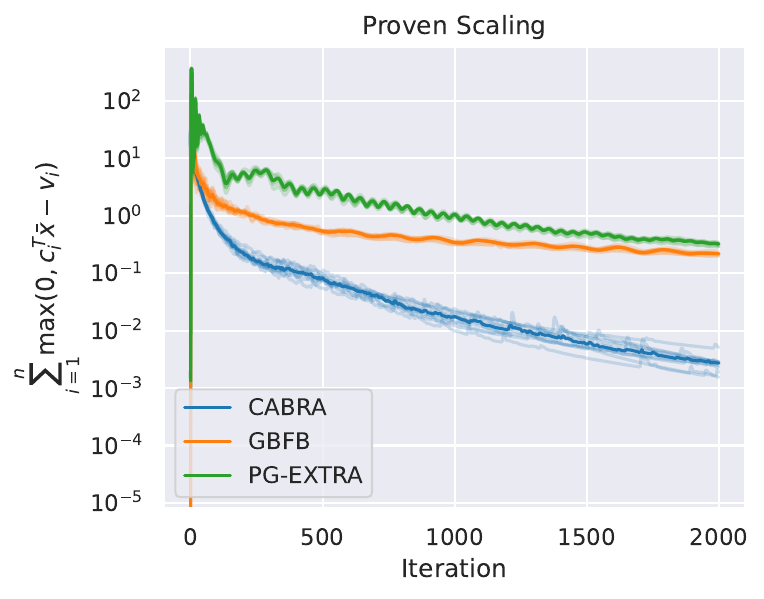} 
        \vspace{0.5em}
        \includegraphics[width=\textwidth]{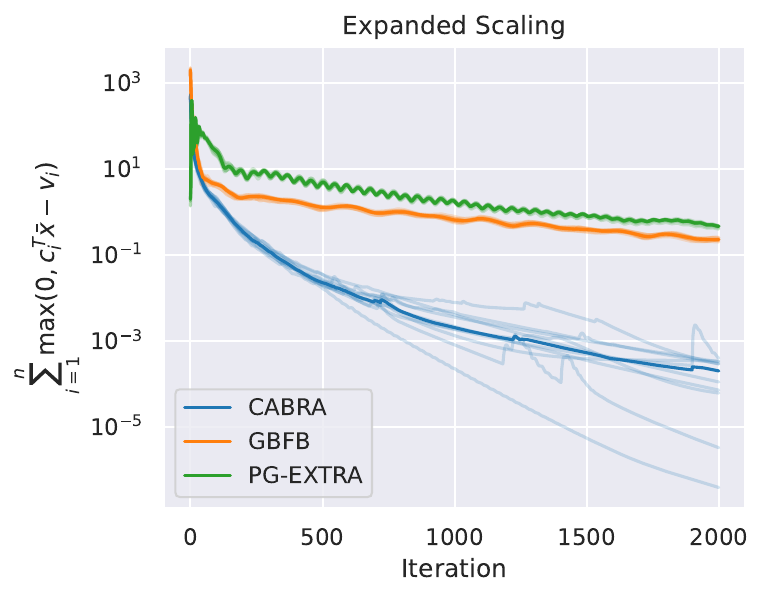} 
        \caption{Constraint violation}
        \label{cab:fig:swta_violation}
    \end{subfigure}
    \hfill
    \begin{subfigure}[b]{0.31\textwidth}
        \centering
        \includegraphics[width=\textwidth]{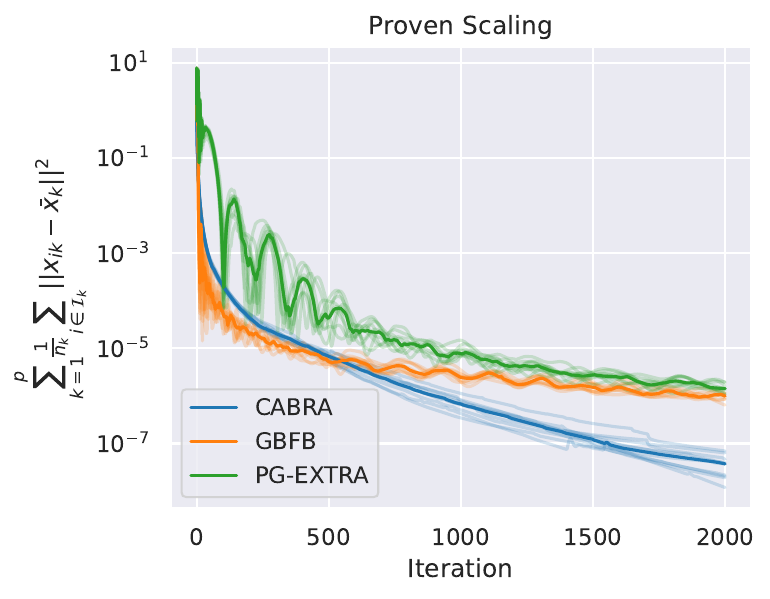} 
        \vspace{0.5em}
        \includegraphics[width=\textwidth]{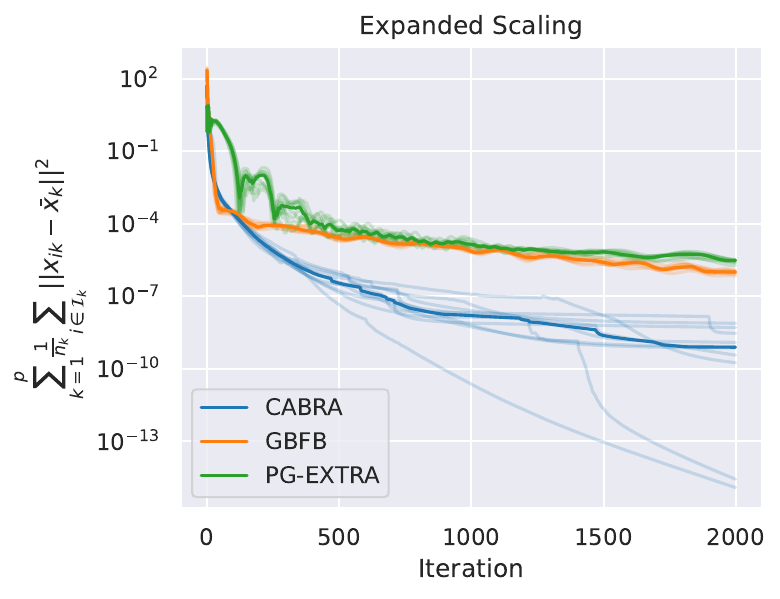} 
        \caption{Consensus error}
        \label{cab:fig:swta_consensus}
    \end{subfigure}

\caption[Multi-stage stochastic weapon target assignment results]{Multi-stage stochastic weapon target assignment results. CABRA outperforms PG-EXTRA and the Graph-Based Forward-Backward Splitting in objective function gap, sum of the squared distances from the mean iterate, and constraint violation.}
\label{cab:fig:swta}
\end{figure}

\section{Conclusion}

This research extends the matrix-parametrized (adaptable) forward-backward splitting framework with minimal lifting by developing CABRA, a family of algorithms which find a zero in the sum of maximal monotone and cocoercive operators composed with selection operators.
It also offers several approaches for selecting valid matrix parameters, and demonstrates their use for parallelization and preconditioning.
We apply CABRA to the continuous relaxation of the multi-stage stochastic weapon target assignment problem, and show that it can be chosen to split over weapons platforms and converges faster than PG-EXTRA and graph-based forward-backward splitting methods, algorithms which do not account for the structure of the selection operators.

The CABRA family of algorithms, and the matrix parameter selection framework, offer a number of promising directions for future research, including its extension to arbitrary linear compositions and to weighted sums of operators, which frequently appear in the context of stochastic programming.
We also believe that further research into optimal preconditioning matrices which balance convergence rate and parallelism will prove fruitful.

\noindent\textbf{Acknowledgements} Both authors acknowledge support from Office of Naval Research awards N0001425WX00069 and N0001425GI01512.

\bibliography{ch4/files/thebib}%

\appendix 
\section{Appendix}
\label{Appendix}

This appendix provides a full proof of the convergence of CABRA as applied to problem \eqref{cab:zero_in_coupled_sum}.
In this context, we use the following enhanced definitions and notation.
For all $i \in \mathcal{I}$, let 
\begin{align*}
\HH_i &= \bigoplus_{k \in \kA_i} \GG_k \;\text{be the product space over the indices in }\; \kA_i,\\
A_i&\colon \HH_i \setto \HH_i \;\text{be a maximal monotone operator, and}\; \\
\RAi{i}&\colon \bigoplus_{k \in \mathcal{K}}\GG_k \to \HH_i \;\text{be a selection operator which selects the subvectors given by }\;\kA_i,
\intertext{
so that $\RAi{i} y = (y_{k})_{k \in \kA_i}$. 
Similarly, for all $j \in \mathcal{J}$, let 
}
\BB_j &= \bigoplus_{k \in \kBj{j}} \GG_k \;\text{be the product space over the indices in }\; \kBj{j}, \\
B_j&\colon \BB_j \to \BB_j \;\text{be a }\; \beta_j\text{-cocoercive operator, and}\\
\RBj{j}&\colon \bigoplus_{k \in \mathcal{K}}\GG_k \to \BB_j \;\text{be a selection operator which selects the subvectors given by }\;\kBj{j}
\end{align*}
with $\RBj{j} y = (y_{k})_{k \in \kBj{j}}$.
Let
\begin{equation*}
\Hx = \bigoplus_{i \in \mathcal{I}} \HH_i\;\text{and} \;
\Bx = \bigoplus_{j \in \mathcal{J}} \BB_j
\end{equation*} 
so that for $\v{x}_i \in \HH_i$ and $\xb_j \in \BB_j$ we have $\v{x} = (\v{x}_i)_{i \in \mathcal{I}} \in \Hx$ and $\xb = (\xb_j)_{j \in \mathcal{J}} \in \Bx$.
We also define lifted operators
\begin{align*}
\v{A}&\colon \Hx \setto \Hx \; \text{such that} \;\v{A}\v{x} = (A_i x_i)_{i \in \mathcal{I}}, \;\text{and}\\
\v{B}&\colon \Bx \to \Bx \; \text{such that} \;\v{B}\xb = (B_j x'_j)_{j \in \mathcal{J}}.
\end{align*}
We define the lifted coupling operators
\begin{align*}
\RA&\colon \bigoplus_{k \in \mathcal{K}}\GG_k \to \Hx \; \text{such that} \; \RA y = (\RAi{i} y)_{i \in \mathcal{I}}, \;\text{and}\\
\RB&\colon \bigoplus_{k \in \mathcal{K}}\GG_k \to \Bx \; \text{such that} \; \RB y = (\RBj{j} y)_{j \in \mathcal{J}}.
\end{align*}
Operators $\RAs$ and $\RBs$ are the adjoint operators of $\RA$ and $\RB$ respectively with $\RAs \v{A}\v{x}$ returning a set containing $\RAs \v{z}$ for each $\v{z} \in \v{A}\v{x}$.
We note that, for $\v{w} \in \Hx$ and $\v{b} \in \Bx$, $\RAs\v{w}$ and $\RBs\v{b}$ implicitly sum over the entries in $\v{w}$ and $\v{b}$ associated with each $y_k$ so that
\begin{subequations}
    \label{cab:RAs_sum}
\begin{align}
    (\RAs \v{w})_k &= \sum_{i \in \mathcal{I}_k} \v{w}_{ik}, \;\text{and} \\
    (\RBs \v{b})_k &= \sum_{j \in \mathcal{J}_k} \v{b}_{jk}. 
\end{align}
\end{subequations}
With the definitions above, we can succinctly rewrite problem \eqref{cab:zero_in_coupled_sum} as 
\begin{equation}
    \label{cab:zero_in_mono_coco_coupled}
\text{find}\; y \in \bigoplus_{k \in \mathcal{K}}\GG_k \;\text{such that}\; 0 \in \RAs\v{A}\RA y + \RBs \v{B}\RB y.
\end{equation}
For each subvector $y_k \in \Gk$, we now form a lifted space $\GG_k^{n_k}$ which contains a copy of $y_k$ for each operator index in $\mathcal{I}_k$, a reduced space $\GG_k^{n_k-1}$ with one less copy of $\GG_k$, and another lifted space $\GG_k^{m_k}$ which contains a copy of $y_k$ for each operator index in $\mathcal{J}_k$.
Let
\begin{equation*}
\Hy = \bigoplus_{k \in \mathcal{K}} \GG_k^{n_k}, \; \Hz = \bigoplus_{k \in \mathcal{K}} \Gk^{n_k-1},\;\text{and}\; \By = \bigoplus_{k \in \mathcal{K}} \GG_k^{m_k} 
\end{equation*}
be product spaces over these lifted spaces.
Thus, for $\v{y}_k \in \GG_k^{n_k}$, $\v{z}_k \in \GG_k^{n_k - 1}$, and $\ybk_k \in \GG_k^{m_k}$, we define lifted variables 
\begin{align*}
\v{y} &= (\v{y}_k)_{k \in \mathcal{K}} \in \Hy, \\
\v{z} &= (\v{z}_k)_{k \in \mathcal{K}} \in \Hz,\;\text{ and} \\ 
\yb &= (\yb_k)_{k \in \mathcal{K}} \in \Hy.
\end{align*}
Our permutation operators, $\PA$ and $\PB$, continue to operate as before, but now on the lifted spaces $\Hy$ and $\By$.
We note that, as permutation operators, the adjoints of $\PA$ and $\PB$, $\PAs \colon \Hx \to \Hy$ and $\PBs \colon \Bx \to \By$, are also the inverses of $\PA$ and $\PB$, respectively.
\subsubsection*{Matrix Parameters}
We now introduce the lifted matrix parameters which serve as the building blocks for our linear operators in these spaces.
As in \cite{bassett2024optimaldesignresolventsplitting,tam2023frugal,bredies2024graph}, our convergence proof rest on showing that the iterations correspond with an averaged operator in a reduced product space.
In those works, the space has one less copy of the decision variable space than the number of operators.
In this research, that property extends to each entry in the product space of the decision variables.
We therefore define the additional set of matrices necessary for working in this reduced product space as $M_k \in \R^{(n_k-1) \times n_k}$ solving the equation $W_k = M_k^\intercal M_k$, which can be found, for example, via Cholesky decomposition. 
This method and others are described in \cite{bassett2024optimaldesignresolventsplitting}.
With $D_k = \diag(Z_k)$ and $L_k = -\tril(Z_k)$, the conditions in \eqref{cab:cabra_reqs} imply that 
\begin{subequations}
\begin{align}
    D_k - L_k - L_k^\intercal &= Z_k\label{cab:Z_def}, \\
    \1_{n_k}^\intercal (D_k - 2L_k)\1_{n_k} &= 0, \label{cab:DL_sum}\\
    U\1_{n_k} &= 0,\;\text{and} \\
    \nullspace(M_k) &= \text{span}(\1_{n_k})\label{cab:M_connect}.
\end{align}
\end{subequations}
When operating on arbitrary real Hilbert spaces $\GG_k$ rather than $\R$ for each element, our next step is to lift the matrices $Q_k$, $K_k$, $U_k$, $Z_k$, $W_k$, $D_k$, $L_k$, $W_k$, and $M_k$ to operate on lifted vectors in their respective product spaces ($\GG_k^{m_k}$, $\GG_k^{n_k}$, or $\GG_k^{n_k - 1}$).
We do so by constructing bounded linear operators of the form $\v{Z}_k\colon \GG_k^{n_k} \to \GG_k^{n_k}$ as $\v{Z}_k = Z_k \otimes \I_k$, where $\I_k$ is the identity operator on $\GG_k$.
Lifting the other operators in a similar fashion, we have the following collection of lifted operators:
\begin{equation}
    \label{cab:cabra_mats}
\begin{aligned}
        \v{Z}_k &\colon \GG_k^{n_k} \to \GG_k^{n_k} \quad&\text{and} &&
        \v{D}_k &\colon \GG_k^{n_k} \to \GG_k^{n_k} \\ 
        \v{Q}_k &\colon \GG_k^{m_k} \to \GG_k^{n_k} \quad& &&
        \v{L}_k &\colon \GG_k^{n_k} \to \GG_k^{n_k} \\
        \v{K}_k &\colon \GG_k^{n_k} \to \GG_k^{m_k} \quad& &&
        \v{U}_k &\colon \GG_k^{n_k} \to \GG_k^{n_k} \\
        \v{W}_k &\colon \GG_k^{n_k} \to \GG_k^{n_k} \quad& & &
        \v{M}_k &\colon \GG_k^{n_k} \to \GG_k^{n_k - 1} .
    \end{aligned}
\end{equation}
In an analogous construction to forming block diagonal matrices, we then define the following lifted operators on $\Hy$ as the direct sum of their corresponding operators in \eqref{cab:cabra_mats}:
\begin{subequations}
    \label{cab:cabra_mats_diag}
\begin{align}
    \v{Z} &\colon \Hy \to \Hy & \v{Z}\v{y} &= (\v{Z}_k \v{y}_k)_{k \in \mathcal{K}}\\
    \v{D} &\colon \Hy \to \Hy & \v{D}\v{y} &= (\v{D}_k \v{y}_k)_{k \in \mathcal{K}}\\
    \v{L} &\colon \Hy \to \Hy & \v{L}\v{y} &= (\v{L}_k \v{y}_k)_{k \in \mathcal{K}}\\
    \v{W} &\colon \Hy \to \Hy & \v{W}\v{y} &= (\v{W}_k \v{y}_k)_{k \in \mathcal{K}}\\
    \v{M} &\colon \Hy \to \Hz & \v{M}\v{y} &= (\v{M}_k \v{y}_k)_{k \in \mathcal{K}}\\
    \v{U} &\colon \Hy \to \Hy & \v{U}\v{y} &= (\v{U}_k \v{y}_k)_{k \in \mathcal{K}}\\
    \v{K} &\colon \Hy \to \By & \v{K}\v{y} &= (\v{K}_k \v{y}_k)_{k \in \mathcal{K}}\\
    \v{Q} &\colon \By \to \Hy & \v{Q}\yb &= (\v{Q}_k \ybk_k)_{k \in \mathcal{K}}.
\end{align}
\end{subequations}
As before, for ease of notation, we define the following permute compositions:
\begin{subequations}
    \label{cab:cabra_mats_perm}
\begin{align}
    \ZA &\colon \Hx \to \Hx & \ZA \v{x} &= \PA \v{Z} \PAs \v{x}\\
    \DA &\colon \Hx \to \Hx & \DA \v{x} &= \PA \v{D} \PAs \v{x} \\
    \LA &\colon \Hx \to \Hx & \LA \v{x} &= \PA \v{L} \PAs \v{x} \\
    \KA &\colon \Hx \to \Bx & \KA \v{x} &= \PB \v{K} \PAs \v{x} \\
    \QA &\colon \Bx \to \Hx & \QA \xb   &= \PA \v{Q} \PBs \xb \\
    \UA &\colon \Hx \to \Hx & \UA \v{x} &= \PA \v{U} \PAs \v{x}\\
    \WA &\colon \Hx \to \Hx & \WA \v{x} &= \PA \v{W} \PAs \v{x}\\
    \MA &\colon \Hx \to \Hz & \MA \v{x} &= \v{M} \PAs \v{x}.
\end{align}
\end{subequations}
\subsubsection*{Notes on the permutation operators and selection operators}
We now state a number of facts relating the permutation operators, the selection operators, their adjoints, and their ranges, which we rely on in our proof.
We note first that, for any $\v{x} = \RA y$, each $\v{x}_{ik}$ equals the same value $y_k$ across all $i \in \mathcal{I}_k$ and for $\v{y} = \PAs\v{x} = \PAs\RA y$, we have $\v{y}_k = \1_{n_k} \otimes y_k$.
For $y \in \Gy$, we define $\v{y}_{\1} = (\1_{n_k} \otimes y_k)_{k \in \mathcal{K}}$ and $\ybk_{\1} = (\1_{m_k} \otimes y_k)_{k \in \mathcal{K}}$.
For $y$, $\v{y}_{\1}$, and $\ybk_{\1}$ we therefore have the following relationship between $\PA$ and $\RA$, and $\PB$ and $\RB$:
\begin{subequations}
    \begin{align}
    \v{y}_{\1} &= \PAs \RA y, \; \text{and}\\
    \ybk_{\1} &= \PBs \RB y.
    \end{align}
\end{subequations}
For $\v{x} \in \Hx$ and $\xb \in \Bx$, we also have, for all $k \in \mathcal{K}$,
\begin{subequations}
    \label{cab:PtoRx}
    \begin{align}
    \sum_{i \in \mathcal{I}_k } (\PAs \v{x})_{ki}  &= \sum_{i \in \mathcal{I}_k } \v{x}_{ik} = (\RAs \v{x})_k \;\text{and}\\
    \sum_{j \in \mathcal{J}_k } (\PBs \xb)_{kj}  &= \sum_{j \in \mathcal{J}_k } \xb_{jk} = (\RBs \xb)_k .
    \end{align}
\end{subequations}
Given the requirement for $K_k \1_{n_k} = \1_{m_k}$ and $Q_k^\intercal \1_{n_k} = \1_{m_k}$ in \eqref{cab:cabra_Q_sum} and \eqref{cab:cabra_K_sum}, we also have
\begin{subequations}
    \label{cab:PQKtoR}
    \begin{align}
    \ybk_{\1} = \v{K} \v{y}_{\1} &= \v{K} \PAs \RA y,\\
    \PB \v{K} \PAs \RA y &= \KA \RA y = \RB y, \;\text{and} \label{cab:PBKPAx}\\
    \sum_{i \in \mathcal{I}_k } (\v{Q} \PBs \xb)_{ki} &= \sum_{j \in \mathcal{J}_k} \xb_{jk} = (\RBs \xb)_k. \label{cab:QPeqRB}
    \end{align}
\end{subequations}
We define $\mathcal{N}_A$ as the range of $\RA$ and $\mathcal{N}_A^\perp$ as its orthogonal complement, noting that
\begin{subequations}
\begin{align}
    \mathcal{N}_A &= \{\v{x} \in \Hx: \exists y \in \Gy \,\text{where}\, \v{x} = \RA y\}\;\text{and} \\ 
    \mathcal{N}^\perp_A &= \{\v{x}^\perp \in \Hx: \langle \v{x}^\perp, \v{x} \rangle = 0 \quad \forall \v{x} \in \mathcal{N}_A\} .
\end{align}
\end{subequations}
Notably, since the orthogonal space is the null space of the adjoint, we have
\begin{equation}
    \label{cab:Nperp}
    \v{x}^\perp \in \mathcal{N}^\perp_A \iff \RAs \v{x}^\perp = 0 \iff \sum_{i \in \mathcal{I}_k} (\PAs \v{x}^\perp)_{ki} = \sum_{i \in \mathcal{I}_k} \v{x}^\perp_{ik} = 0 \quad \forall k \in \mathcal{K}.
\end{equation}
The definition of $\PA$ and $\mathcal{N}_A$, and the requirement for the null space of $W_k$ and $Z_k$ to be the ones vector, together imply that 
\begin{subequations}
    \label{cab:lemmaclaims}
\begin{align}
    \nullspace(\MA) &= \nullspace{(\WA)} = \nullspace{(\ZA)} = \mathcal{N}_A , \\
    \range(\MT) &= \mathcal{N}_A^\perp,\;\text{and} \\
    (\DA - 2\LA)\v{x} &\in \mathcal{N}_A^\perp \quad \forall \v{x} \in \mathcal{N}_A, \label{cab:DLA}
\end{align}
\end{subequations}
which we establish in Lemmas \ref{cab:Lem:nullspace} and \ref{cab:Lemma:DL1_W}.

\begin{lemma} \label{cab:Lem:nullspace}
    The null space of $\MA$, $\WA$, and $\ZA$ is $\mathcal{N}_A$.
\end{lemma}
\begin{proof}
    We begin by noting that, by the requirements of \eqref{cab:cabra_reqs} on $W_k$ and $Z_k$, and the definition of $M_k$,
    \[
    M_k \1_{n_k} = W_k \1_{n_k} = Z_k \1_{n_k} = 0.
    \]
    For any $\v{x} \in \mathcal{N}_A$, there exists a $y \in \Gy$ such that $\PAs\v{x} = \left(\1_{n_k} \otimes y_k\right)_{k \in \mathcal{K}}$. 
    This follows from the construction of $\PA$ and the definition of $\mathcal{N}_A$.
    We therefore have
    \begin{alignat*}{3}
    \WA\v{x} &= \PA \v{W} \PAs \v{x} &&= \PA \left(W_k \1_{n_k} \otimes y_k \right)_{k \in \mathcal{K}}  &&= 0, \\
    \ZA\v{x} &= \PA \v{Z} \PAs \v{x} &&= \PA \left(Z_k \1_{n_k} \otimes y_k \right)_{k \in \mathcal{K}}  &&= 0 , \; \text{and}\\
    \MA\v{x} &= \PA \v{M} \PAs \v{x} &&= \PA \left(M_k \1_{n_k} \otimes y_k \right)_{k \in \mathcal{K}}  &&= 0. 
    \end{alignat*}
    We therefore have $
        \mathcal{N}_A \subseteq \nullspace{\MA}, \; \mathcal{N}_A \subseteq \nullspace{\WA}, \; \text{and}\; \mathcal{N}_A \subseteq \nullspace{\ZA}$. 
    We know by \eqref{cab:cabra_reqs} that $\nullspace{W_k} = \nullspace{M_k} = \nullspace{Z_k} = \text{span}(\1_{n_k})$.
    For any $\v{x} \notin \mathcal{N}_A$, let $\v{y} = \PAs \v{x} \in \Hy$.
    We note that $\v{y}_k \neq \1 \otimes y_k$ for some $k \in \mathcal{K}$ (otherwise $\v{x}$ would be in $\mathcal{N}_A$), and $\v{y}_k$ therefore is not in the null space of $\v{M}_k$, $\v{W}_k$, or $\v{Z}_k$.
    We therefore have $\v{M}_k \v{y}_k \neq 0 $, $\v{W}_k \v{y}_k \neq 0$, and $\v{Z}_k \v{y}_k \neq 0$, and
    $
        \nullspace{\MA} = \nullspace{\WA} = \nullspace{\ZA} = \mathcal{N}_A .%
    $
\end{proof}
A direct corollary of Lemma \ref{cab:Lem:nullspace} is the fact that 
\begin{equation}
    \range(\MT) = \mathcal{N}_A^\perp \label{cab:range_PM}.
\end{equation}
For $\v{x} \in \mathcal{N}_A$, there is a similar correspondence between $\PAs(\v{D} - 2\v{L})\PA\v{x}$ and $\mathcal{N}_A^\perp$, which we establish in Lemma \ref{cab:Lemma:DL1_W}.
\begin{lemma}    \label{cab:Lemma:DL1_W}
    If $\v{x} \in \mathcal{N}_A$ then $(\DA - 2\LA)\v{x} \in \mathcal{N}_A^\perp$.
\end{lemma}
\begin{proof}
    Let $\v{x} \in \mathcal{N}_A$.
    We have already established by definition that
    \[\PAs\v{x}  = \left(\1_{n_k} \otimes y_k\right)_{k \in \mathcal{K}}\]
    for some subvectors $y_k \in \Gk$.
    We therefore know that
    \begin{equation*}
        (\v{D} - 2\v{L})\PAs\v{x} = \left((D_k - 2L_k)\1_{n_k} \otimes y_k\right)_{k \in \mathcal{K}}.
    \end{equation*}
    By assumption \eqref{cab:cabra_Z_sum}, we know that each $Z_k$ is symmetric and $Z_k\1_{n_k}=0$, so $\1_{n_k}^\intercal Z_k \1_{n_k} = 0$.
    We also know by assumption \eqref{cab:Z_def} that 
    \begin{align*}
    Z_k &= D_k - L_k - L_k^\intercal, \;\text{and therefore}\\
    \1_{n_k}^\intercal (D_k - 2L_k)\1_{n_k} &= \1_{n_k}^\intercal Z_k \1_{n_k} = 0.
    \end{align*}
    This means that for each $k$, for any $y^1_k \in \Gk$, \[\left\langle \1_{n_k} \otimes y^1_k, (\v{D}_k - 2\v{L}_k)(\1_{n_k} \otimes y_k)\right\rangle = 0.\] 
    For any $\v{x}^1 \in \mathcal{N}_A$ and $\PAs\v{x}^1 = \left(\1 \otimes y_k^1\right)_{k \in \mathcal{K}}$, we then have
    \begin{align*}    
        \left\langle \v{x}^1, \PA (\v{D} - 2\v{L})\PAs\v{x} \right\rangle &= \left\langle \PAs\v{x}^1, (\v{D} - 2\v{L})\PAs\v{x} \right\rangle \\
        &= \sum_{k=1}^p \left\langle \1_{n_k} \otimes y_k^1, (\v{D}_k - 2\v{L}_k)(\1_{n_k} \otimes y_k) \right\rangle \\
        &= 0.
    \end{align*}
    Since $\mathcal{N}_A^\perp$ is defined as the set of vectors whose inner product with any member of $\mathcal{N}_A$ is zero, we therefore have 
    \begin{equation}\label{cab:LIx_in_NT}
    \PA(\v{D} - 2\v{L})\PAs\v{x} \in \mathcal{N}_A^\perp \quad \forall \v{x} \in \mathcal{N}_A.
    \end{equation}
\end{proof}

\subsection*{Reduced Algorithm Definition}
\begin{algorithm} 
    \caption{CABRA (reduced form)}\label{cab:cabra_d_alg} 
    \begin{algorithmic}[1] 
    \Require $\alpha \in (0, 4)$; $\gamma_\nu \in (0, 2 - \frac{\alpha l}{2})$; $\v{z}^0 \in \bigoplus_{k \in \mathcal{K}}\GG_k^{n_k-1}$; $\DA^{-1}$, $\MA$, $\LA$, $\QA$, $\KA$ satisfying \eqref{cab:cabra_reqs}, \eqref{cab:cabra_mats}, \eqref{cab:cabra_mats_diag}, and \eqref{cab:cabra_mats_perm}.
    \State $\nu \gets 0$
    \Repeat
      \State $\v{x}^\nu = J_{\alpha \DA^{-1}\v{A}}\left(\DA^{-1}\left[-\MT \v{z}^\nu + 2\LA \v{x}^\nu - \alpha \QA\v{B}\KA\v{x}^\nu\right]\right)$ \label{cab:cabra_d_itr1}
      \State $\v{z}^{\nu+1} = \v{z}^\nu + \gamma_\nu \MA\v{x}^\nu$ \label{cab:cabra_d_itr2}
      \State $\nu \gets \nu+1$
    \Until{convergence}
\end{algorithmic}
\end{algorithm}

We now introduce the fixed-point iteration which defines the reduced form of CABRA given by Algorithm \ref{cab:cabra_d_alg}.
Let the operators $S\colon \Hz \to \Hz$ and $T\colon \Hz \to \Hx$ be given by
\begin{subequations}
    \label{cab:cabra_operators}
\begin{align}
    S(\v{z}) &= \left(\DA - 2\LA + \alpha \v{A} + \alpha \QA \v{B} \KA\right)^{-1}(- \MT \v{z}) \\
    T(\v{z}) &= \left(\I + \gamma \MA S\right)(\v{z}).
\end{align}
\end{subequations}
This is equivalent to $T(\v{z}) = \v{z} + \gamma \MA\v{x}$ where $\v{x}$ is defined as in Algorithm \ref{cab:cabra_d_alg} step \eqref{cab:cabra_d_itr1}.
We note that, by design, the permutation $\PA$ preserves the ordering of the operators within each $L_k$, making $\LA$ strictly block lower triangular so that the block matrix $(\LA\v{x})_i$ depends only on $(x_1, \dots, x_{i-1})$.
We also note that $(\QA\v{B}\KA\v{x})_i$ depends only on $(x_1, \dots, x_{i-1})$ as well, due to the structure imposed on $Q_k$ and $K_k$ by \eqref{cab:cutoff_assumptions}.
Therefore $\v{x}$ in \eqref{cab:cabra_d_itr1} is well-defined, and can be found by iteratively finding $J_{\alpha (\DA^{-1})_i A_i}$ via forward substitution, working in parallel when any set of resolvents have all the input required by $\LA$ and $\QA\v{B}\KA$.

\subsection*{Proof of convergence}
We now proceed to prove the convergence of CABRA. 
In Lemma \ref{cab:Lemma:fixed_point_corresponds_coupled}, we begin by showing a correspondence between fixed points of $T$ and zeros of problem \eqref{cab:zero_in_coupled_sum}.
We then establish a result linking the operator $\v{Q}\PBs\v{B}\PB\v{K}$ to $\v{U}$ (and therefore to $\v{Z}$ and $\v{W}$), which will be necessary for establishing the nonexpansivity of $T$.
We then show the averaged nonexpansivity of $T$, and, given this averaged nonexpansivity, we then proceed to show that the existence of a zero is sufficient for weak convergence of the iterates $(\v{z}^\nu)$, and that the corresponding series $(\v{x}^\nu)$ weakly converges to a point $\v{x}^* \in \mathcal{N}_A$ which corresponds with a solution to \eqref{cab:zero_in_coupled_sum}.
\begin{lemma} \label{cab:Lemma:fixed_point_corresponds_coupled} %
    There is a direct correspondence between solutions to \eqref{cab:zero_in_coupled_sum} and fixed points of $T$. 
    In particular, $\mathrm{Fix}(T) \neq \emptyset \iff \zer{\RAs\v{A}\RA + \RBs \v{B}\RB} \neq \emptyset$, with the relationship given by
    \begin{equation} 
        y \in \zer{\RAs\v{A}\RA + \RBs \v{B}\RB} \quad \Longleftrightarrow \quad \exists\, \v{z} \in \mathrm{Fix}(T) \text{ such that } \RA y = S \v{z}.
    \end{equation}
\end{lemma}

\begin{proof}
    If $y \in \zer{\RAs\v{A}\RA + \RBs \v{B}\RB}$, then there exist a corresponding $\v{x}, \v{w} \in \Hx$, and $\xb, \v{u} \in \Bx$ such that 
    \begin{subequations}
    \begin{align}
        \v{x} &= \RA y \in \mathcal{N}_A, \\
        \xb &= \RB y = \KA \v{x} ,\label{cab:lem_k_sum}\\
        \v{w} &\in \v{A}\v{x} ,\\ %
        \v{u} &= \v{B}\xb = \v{B}\KA \v{x} ,\;\text{and}\\ %
        0 &= \RAs \v{w} + \RBs \v{u} \label{cab:RAswRBsu}.
    \end{align}
    \end{subequations}
    Equation \eqref{cab:lem_k_sum} follows from assumption \eqref{cab:PBKPAx}. 
    Looking at the individual subvectors of \eqref{cab:RAswRBsu}, for each $k \in \mathcal{K}$ we have 
    \begin{subequations}
    \begin{align}
        0 &= \left(\RAs \v{w} + \RBs \v{u}\right)_k  \\
        &= \sum_{i \in \mathcal{I}_k} (\PAs \v{w})_{ki} + \sum_{i \in \mathcal{I}_k} (\v{Q} \PBs \v{u})_{ki} \label{cab:lem_q_sum} \\
        &= \sum_{i \in \mathcal{I}_k} (\PAs \v{w}  + \v{Q}\PBs \v{u})_{ki} \\
        &= \sum_{i \in \mathcal{I}_k} (\PAs \v{w} + \PAs \PA \v{Q}\PBs \v{u})_{ki} \\
        &= \sum_{i \in \mathcal{I}_k} \left(\PAs \left[\v{w} + \QA \v{u}\right]\right)_{ki},
    \end{align}
    \end{subequations}
    where \eqref{cab:lem_q_sum} follows from \eqref{cab:QPeqRB}.
    Thus, by \eqref{cab:Nperp}, \[\v{w} + \QA \v{u} \in \mathcal{N}^\perp_A.\]
    Since $\v{x} \in \mathcal{N}_A$, we also know by Lemma \ref{cab:Lemma:DL1_W} that
    \[ (\DA - 2\LA) \v{x} \in \mathcal{N}_A^\perp.\]
    Thus
    \begin{equation*}        
    \alpha \v{w} + \alpha \QA \v{u}  + (\DA - 2\LA) \v{x} \in \mathcal{N}_A^\perp .
    \end{equation*}
    We also know by Lemma \ref{cab:Lem:nullspace} that $\range(\MT) = \mathcal{N}_A^\perp$.
    Therefore there exists some $\v{z} \in \Hz$ such that 
    \[
        -\MT \v{z} = \alpha \v{w} + \alpha \QA \v{u}  + (\DA - 2\LA) \v{x}. \]
    Therefore, by the definition of $\v{w}$ and $\v{u}$,
    \begin{align*}
    \alpha \v{A} \v{x} &\ni -\MT \v{z} - \alpha \QA \v{B}\KA\v{x} + (2\LA - \DA) \v{x} ,\;\text{and} \\
    \v{x} + \alpha \DA^{-1} \v{A}\v{x} &\ni \DA^{-1} \left(-\MT \v{z} - \alpha \QA \v{B}\KA\v{x} + 2\LA \v{x}\right).
    \end{align*}
    This means that $\v{x}$ satisfies 
    \begin{equation*}
        \v{x} = J_{\alpha \DA^{-1}\v{A}} \left(\DA^{-1}\left[-\MT \v{z} - \alpha \QA \v{B}\KA\v{x} + 2\LA \v{x}\right]\right),
    \end{equation*}
    and by the uniqueness of the resolvent of a maximal monotone operator, we have $\v{x} = S(\v{z})$.
    Since $\v{x} \in \mathcal{N}_A$ and $\nullspace(\MA) = \mathcal{N}_A$, we have
    \begin{align*}
    T(\v{z}) &= \v{z} + \gamma \MA \v{x} = \v{z}\\
    \implies \v{z} &\in \mathrm{Fix}(T)
    \end{align*}
    proving the forward implication.

    If $\v{z} \in \mathrm{Fix}(T)$ with $\v{x} = S \v{z}$, we know that $\v{x} \in \mathcal{N}_A$. 
    Therefore there exists $y \in \bigoplus_{k \in \mathcal{K}}\GG_k$ such that $\v{x} = \RA y$.
    Let $\v{u} = \v{B} \RB y$.    
    We know by \eqref{cab:PBKPAx} that $\KA \v{x} = \RB y$, so we also have $\v{u} = \v{B} \KA \v{x} $.
    By the definition of $S$, we also know that 
    \begin{subequations}
    \begin{align*}
        \alpha \v{A} \v{x} &\ni -\MT \v{z} - \alpha \QA \v{B}\KA\v{x} + (2\LA - \DA) \v{x}  \\
        \alpha \left(\v{A} \v{x} + \QA \v{u}\right)&\ni -\MT \v{z} + (2\LA - \DA) \v{x}  .
    \end{align*}
    \end{subequations}
    Since, by Lemmas \ref{cab:Lem:nullspace} and \ref{cab:Lemma:DL1_W}, $\MT \v{z} \in \mathcal{N}^\perp_A$ and $(2\LA - \DA) \v{x} \in \mathcal{N}_A^\perp$, we know that there must exist $\v{w} \in \v{A} \v{x}$ such that
    \begin{equation*}
        \alpha (\v{w} + \QA \v{u}) \in \mathcal{N}_A^\perp.
    \end{equation*}
    Therefore by \eqref{cab:Nperp}, \eqref{cab:PtoRx}, and \eqref{cab:PQKtoR}, for all $k \in \mathcal{K}$, we have
    \begin{subequations}
    \begin{align*}
        \sum_{i \in \mathcal{I}_k}   \left(\alpha\PAs \left[\v{w} + \QA \v{u}\right]\right)_{ki}&= 0\\
        \sum_{i \in \mathcal{I}_k}   \left(\PAs \v{w}\right)_{ki} + \sum_{i \in \mathcal{I}_k}   \left(\v{Q}\PBs \v{u}\right)_{ki} &= 0 \\
        (\RAs \v{w})_k + \sum_{j \in \mathcal{J}_k} (\PBs \v{u})_{kj} &= 0 \\
        (\RAs \v{w})_k + (\RBs \v{u})_k &= 0 ,
    \end{align*}        
    \end{subequations}
    so that
    \begin{subequations}
    \begin{align*}
        \RAs \v{w} + \RBs \v{u} &= 0\\
        \RAs \v{A}\RA y + \RBs \v{B} \RB y &\ni 0,\;\text{and}\\
        \zer{\RAs\v{A}\RA + \RBs \v{B}\RB} &\ni y ,
    \end{align*}        
    \end{subequations}
    concluding the proof.
\end{proof}

We next establish a result linking the operator $\v{Q} \PBs \v{B}\PB\v{K}$ to $\v{U}$ (and therefore to $\v{W}$), which will be necessary for establishing the nonexpansivity of $T$.
\begin{lemma}\label{cab:Lemma:PBK_coupled} %
The operator $\v{Q}\PBs\v{B}\PB\v{K} + \frac{1}{4}\v{U}$ is maximal monotone.
\end{lemma}

\begin{proof}
    We begin by showing monotonicity.
    For $\v{y}_1, \v{y}_2 \in \Hy$, let $\Delta \v{y} = \v{y}_1 - \v{y}_2$ and $\Delta \v{b} = \v{B}\PB\v{K} \v{y}_1 - \v{B}\PB\v{K} \v{y}_2$.
    By definition, we have
    \begin{align*}
        \left\langle \v{y}_1 - \v{y}_2, \v{Q}\PBs\v{B}\PB\v{K} \v{y}_1 - \v{Q}\PBs\v{B}\PB\v{K} \v{y}_2 \right\rangle =& \left\langle \Delta \v{y}, \v{Q}\PBs\Delta\v{b}\right\rangle \\
        =& \left\langle \PB\v{Q}^* \Delta \v{y}, \Delta\v{b}\right\rangle .
    \end{align*}
    Adding and subtracting $\left\langle \PB\v{K} \Delta \v{y}, \Delta\v{b}\right\rangle$, we have
    \begin{equation*}
       \left\langle \PB\v{Q}^* \Delta \v{y}, \Delta\v{b}\right\rangle  = \left\langle \PB(\v{Q}^* - \v{K}) \Delta \v{y}, \Delta\v{b}\right\rangle + \left\langle \PB\v{K} \Delta \v{y}, \Delta\v{b}\right\rangle.%
    \end{equation*}
    Writing $(\PB \v{K}\v{y}_1)_j$ as $\xb_{1j}$ and  $(\PB \v{K}\v{y}_2)_j$ as $\xb_{2j}$, we know that for all $j$, by the cocoercivity of $B_j$, we have
    \begin{align*}
        \langle \xb_{1j} - \xb_{2j}, B_j \xb_{1j} - B_j \xb_{2j} \rangle &\geq \beta_j \norm{B_j \xb_{1j} - B_j \xb_{2j}}^2\\
        \langle \xb_{1j} - \xb_{2j}, B_j \xb_{1j} - B_j \xb_{2j} \rangle &\geq \norm{\sqrt{\beta_j}\left(B_j \xb_{1j} - B_j \xb_{2j}\right)}^2  %
    \end{align*}
    Define $\beta_{\1}\colon \Bx \to \Bx$ as $\beta_{\1} = (\beta_j \I_j)_{j \in \mathcal{J}}$ where $\I_j$ is the identity in $\BB_j$.
    Based on the definition of $\Bx$, summing over the equations above, we then have
    \begin{align*}
        \left\langle \PB \v{K}\v{y}_1 - \PB \v{K}\v{y}_2, \v{B} \PB \v{K}\v{y}_1 - \v{B} \PB \v{K}\v{y}_2\right\rangle &\geq \norm{\beta_{\1}^{\frac{1}{2}}(\v{B}\PB\v{K}\v{y}_1 - \v{B}\PB \v{K}\v{y}_2)}^2 \\
        \langle \PB \v{K}\Delta \v{y}, \Delta\v{b}\rangle &\geq \norm{\beta_{\1}^{\frac{1}{2}}\Delta\v{b}}^2.
    \end{align*}
    Therefore
    \begin{equation*}
        \left\langle \Delta \v{y}, \v{Q} \PBs \Delta\v{b}\right\rangle \geq \left\langle \PB(\v{Q}^* - \v{K}) \Delta \v{y}, \Delta\v{b}\right\rangle + \norm{\beta_{\1}^{\frac{1}{2}}\Delta\v{b}}^2
    \end{equation*}
    and 
    \begin{align}
        \left\langle \Delta \v{y} , \v{Q}\PBs\Delta\v{b}\right\rangle + \frac{1}{4}\left\langle \Delta \v{y}, \v{U}\Delta\v{y}\right\rangle \geq \left\langle \PB (\v{Q}^* - \v{K}) \Delta \v{y}, \Delta\v{b}\right\rangle + \norm{\beta_{\1}^{\frac{1}{2}}\Delta\v{b}}^2 + \frac{1}{4}\left\langle \Delta \v{y}, \v{U}\Delta\v{y}\right\rangle .\label{cab:lemma2_expanded_coupled}
    \end{align}
    Using the definition of $\v{U}$ and the fact that $(\PBs \beta_{\1}^{-1} \PB)_k = \beta_k^{-1} \otimes \I_k$, where $\beta_k$ is as defined in \eqref{cab:beta_k} and $\I_k$ is the identity on $\Gk$, the right-hand side of \eqref{cab:lemma2_expanded_coupled} becomes
    \begin{align*}
        =&\left\langle \PB (\v{Q}^* - \v{K}) \Delta \v{y}, \Delta\v{b}\right\rangle + \norm{\beta_{\1}^{\frac{1}{2}}\Delta\v{b}}^2 + \norm{\frac{1}{2}\beta_{\1}^{-\frac{1}{2}}\PB (\v{Q}^* - \v{K})\Delta \v{y}}^2 \\
        =&\norm{\beta_{\1}^{\frac{1}{2}}\Delta \v{b} + \frac{1}{2}\beta_{\1}^{-\frac{1}{2}}\PB(\v{Q}^* - \v{K})\Delta \v{y}}^2 \geq 0.
    \end{align*}
    Therefore 
    \begin{equation*}
        \left\langle \Delta \v{y} , \v{Q}\PBs\Delta\v{b}\right\rangle + \frac{1}{4}\left\langle \Delta \v{y}, \v{U}\Delta\v{y}\right\rangle \geq 0
    \end{equation*}
    and $\v{Q}\PBs\v{B}\PB \v{K}+ \frac{1}{4}\v{U}$ is monotone.
    The boundedness and linearity of $\v{Q}\PBs$, $\PB\v{K}$, and $\v{U}$, and the cocoercivity of $\v{B}$ imply the continuity of $\v{Q}\PBs\v{B}\PB \v{K}+ \frac{1}{4}\v{U}$.
    Therefore $\v{Q}\PBs\v{B}\PB \v{K}+ \frac{1}{4}\v{U}$ is both continuous and monotone, and is therefore maximal monotone by \cite[Corollary 20.28]{bauschke_combettes}. %
\end{proof}
A direct result of Lemma \ref{cab:Lemma:PBK_coupled} for the permuted case in $\Hx$ is that the permuted operator given by $\QA\v{B}\KA + \frac{1}{4}\UA$ is also maximal monotone.
We now establish that $T$ is averaged nonexpansive.
\begin{lemma}\label{cab:Lemma:Nonexpansive_coupled} %
    $T(\v{z})$ is $\frac{\gamma}{2-\frac{\alpha l}{2}}$-averaged nonexpansive for $\gamma \in (0, 2 - \frac{\alpha l}{2})$.
\end{lemma}
\begin{proof}
We know, by the definition of the resolvent, that for $\v{z}^1, \v{z}^2 \in \Hz$ and $\v{x}^1, \v{x}^2 \in \Hx$, we have
    \begin{align*}
        &\v{x}^i = J_{\alpha \DA^{-1}\v{A}}\left(\DA^{-1}\left[-\MT \v{z}^i + 2\LA\v{x}^i - \alpha \QA\v{B}\KA\v{x}^i\right]\right)\\
        \implies & \v{x}^i + \alpha \DA^{-1}\v{A}\v{x}^i \ni \DA^{-1}\left[-\MT \v{z}^i + 2\LA\v{x}^i - \alpha \QA\v{B}\KA\v{x}^i\right] \\
        \implies & \DA\v{x}^i + \alpha \v{A}\v{x}^i \ni -\MT \v{z}^i + 2\LA\v{x}^i - \alpha \QA\v{B}\KA\v{x}^i \\
        \implies & \alpha \v{A}\v{x}^i \ni -\MT \v{z}^i + (2\LA - \DA)\v{x}^i - \alpha \QA\v{B}\KA\v{x}^i \\
        \implies & \alpha\v{A}\v{x}^i \ni \PA \left[-\v{M}^*\v{z}^i + (2\v{L} -\v{D})\PAs\v{x}^i - \alpha \v{Q}\PBs\v{B}\PB\v{K}\PAs\v{x}^i\right].
    \end{align*}
    Let $\v{y}^1 = \PAs \v{x}^1$, $\v{y}^2 = \PAs \v{x}^2$, $\Delta\v{z} = \v{z}^1 - \v{z}^2$, $\Delta\v{x} = \v{x}^1 - \v{x}^2$, $\Delta\v{y} = \v{y}^1 - \v{y}^2$, $\Delta\v{b} = \v{B}\PB \v{K}\v{y}_1 - \v{B}\PB \v{K}\v{y}_2$, and $\Delta\v{z}^+ = T(\v{z}^1)-T(\v{z}^2)$.
    By the monotonicity of $\v{A}$ (and therefore of $\alpha\v{A}$) and the linearity of $\v{M}$, $\v{L}$, $\v{D}$, and $\v{Q}$ we have:
    \begin{subequations}
    \begin{align} 
    0 &\leq \left\langle \Delta\v{x}, \PA \left[-\v{M}^* \Delta\v{z} + (2\v{L} -\v{D})\PAs \Delta\v{x} - \alpha \v{Q}\PBs\Delta \v{b}\right] \right\rangle \\ %
    0 &\leq \left\langle \Delta\v{y}, -\v{M}^* \Delta\v{z} + (2\v{L} - \v{D})\Delta\v{y} - \alpha \v{Q}\PBs\Delta \v{b}\right\rangle. \label{cab:first_coupled} 
    \end{align}
    \end{subequations}
    Considering just the right-hand side of the inequality \eqref{cab:first_coupled} and symmetrizing the quadratic form $\left\langle \Delta\v{y}, (2\v{L} - \v{D}) \Delta\v{y} \right\rangle$ in light of assumption \eqref{cab:Z_def}, we have the following simplification,
    \begin{subequations}
    \begin{align}
        =& \left\langle -\v{M}\Delta\v{y}, \Delta\v{z} \right\rangle - \alpha\left\langle \Delta\v{y}, \v{Q}\PBs\Delta \v{b} \right\rangle + \left\langle \Delta\v{y}, (2\v{L} - \v{D}) \Delta\v{y} \right\rangle\\
        =& \left\langle -\v{M}\Delta\v{y}, \Delta\v{z} \right\rangle - \alpha\left\langle \Delta\v{y}, \v{Q}\PBs\Delta \v{b} \right\rangle - \left\langle \Delta\v{y}, \v{Z} \Delta\v{y} \right\rangle. \label{cab:oneandhalfcoupled} %
    \end{align}
    \end{subequations}
    By the results of Lemma \ref{cab:Lemma:PBK_coupled} and assumptions \eqref{cab:Z_succeq_W} and \eqref{cab:cabra_Z_succeq_U}, expression \eqref{cab:oneandhalfcoupled} is 
    \begin{subequations}
    \begin{align}
        \leq& \left\langle -\v{M}\Delta\v{y}, \Delta\v{z} \right\rangle + \frac{\alpha}{4}\left\langle \Delta\v{y}, \v{U}\Delta\v{y} \right\rangle - \left\langle \Delta\v{y}, \v{Z} \Delta\v{y}\right\rangle \\
        \leq& \left\langle -\v{M}\Delta\v{y}, \Delta\v{z} \right\rangle + \frac{\alpha l}{4}\left\langle \Delta\v{y}, \v{Z}\Delta\v{y} \right\rangle  - \left\langle \Delta\v{y}, \v{Z} \Delta\v{y} \right\rangle \\
        \leq& \left\langle -\v{M}\Delta\v{y}, \Delta\v{z} \right\rangle - \left(1 - \frac{\alpha l}{4}\right)\left\langle \Delta\v{y}, \v{Z}\Delta\v{y} \right\rangle \\
        \leq& \left\langle -\v{M}\Delta\v{y}, \Delta\v{z} \right\rangle - \left(1 - \frac{\alpha l}{4}\right)\left\langle \Delta\v{y}, \v{W}\Delta\v{y} \right\rangle \label{cab:second_coupled}.
    \end{align}
    \end{subequations}
    Combining \eqref{cab:first_coupled} and \eqref{cab:second_coupled}, we therefore have
    \begin{subequations} 
    \begin{align}
        0 &\leq \left\langle -\v{M}\Delta\v{y}, \Delta\v{z} \right\rangle - \left(1 - \frac{\alpha l}{4}\right)\left\langle \Delta\v{y}, \v{W}\Delta\v{y} \right\rangle  \\
        &\leq \left\langle -\v{M}\Delta\v{y}, \Delta\v{z} \right\rangle - \left(1 - \frac{\alpha l}{4}\right)\norm{ \v{M}\Delta\v{y}}^2  \\
        &\leq \frac{1}{\gamma}\left\langle \Delta\v{z} - \Delta\v{z}^+, \Delta\v{z} \right\rangle - \left(1 - \frac{\alpha l}{4}\right)\frac{1}{\gamma^2}\norm{ \Delta\v{z} - \Delta\v{z}^+}^2  \label{cab:third_coupled}
    \end{align}
    \end{subequations}
    where \eqref{cab:third_coupled} follows from the definition of $\Delta\v{z}^+$.
    By the parallelogram law this is equivalent to
    \begin{align*}
        0 &\leq \frac{1}{2\gamma}\left[ \norm{\Delta\v{z}}^2 + \norm{\Delta\v{z} - \Delta\v{z}^+}^2 - \norm{\Delta\v{z}^+}^2 \right] - \left(1 - \frac{\alpha l}{4}\right)\frac{1}{\gamma^2}\norm{ \Delta\v{z} - \Delta\v{z}^+}^2  \\
        0 &\leq \frac{1}{2\gamma}\left[\norm{\Delta\v{z}}^2 + \frac{\gamma - 2 + \frac{\alpha l}{2}}{\gamma}\norm{ \Delta\v{z} - \Delta\v{z}^+}^2 - \norm{\Delta\v{z}^+}^2\right].
    \end{align*}
    We therefore have 
    \begin{equation*}
        \norm{\Delta\v{z}}^2 + \frac{\gamma - 2 + \frac{\alpha l}{2}}{\gamma}\norm{ \Delta\v{z} - \Delta\v{z}^+}^2 \geq \norm{\Delta\v{z}^+}^2.
    \end{equation*}
    The operator $T$ is therefore $\frac{\gamma}{2-\frac{\alpha l}{2}}$-averaged for $\gamma \in (0, 2 - \frac{\alpha l}{2})$.
\end{proof}
This allows us to prove the convergence of CABRA.
\begin{theorem}
    \label{cab:main_theorem}
Let $(\v{z}^\nu)$ be the sequence of iterates produced by Algorithm \ref{cab:cabra_d_alg} for some sequence of step sizes $(\gamma_\nu) \in (0, 2 - \frac{\alpha l}{2})$ with $\sum_{\nu} \gamma_\nu (2 - \frac{\alpha l }{2} - \gamma_\nu) = +\infty$ applied to problem \eqref{cab:zero_in_coupled_sum}, and $(\v{x}^\nu)$ be the corresponding sequence of solutions such that $\v{x}^\nu = S(\v{z}^\nu)$. 
If the set of solutions of \eqref{cab:zero_in_coupled_sum} is non-empty, then $(\v{z}^\nu)$ converges weakly to a fixed point $\v{z}^*$ such that $S(\v{z}^*) = \RA y^*$ where $y^*$ is a solution to \eqref{cab:zero_in_coupled_sum}, and the associated iterates $(\v{x}^\nu)$ converge weakly to $S(\v{z}^*)$. %
\end{theorem}

\begin{proof}
    Since the set of solutions to \eqref{cab:zero_in_coupled_sum} is non-empty, $T$ has a fixed point by Lemma \ref{cab:Lemma:fixed_point_corresponds_coupled}.
    It is also averaged nonexpansive by Lemma \ref{cab:Lemma:Nonexpansive_coupled}.
    Therefore \cite[Theorem 5.15]{bauschke_combettes} gives that for any starting point $\v{z}^0$ and sequence $(\v{z}^\nu)$ defined by $\v{z}^{\nu+1} = T(\v{z}^\nu)$, $\MA \v{x} \to 0$, $\v{z}^{\nu+1} - \v{z}^\nu \to 0$, and $(\v{z}^\nu)$ converges weakly to some $\v{z}^* \in \mathrm{Fix}(T)$, and by Lemma \ref{cab:Lemma:fixed_point_corresponds_coupled} $S(\v{z}^*) = \RA y^*$ for $y^* \in \zer{\RAs\v{A}\RA + \RBs \v{B}\RB}$, establishing the first result. %

    We now show that $(\v{x}^\nu)$ converges weakly to $S(\v{z}^*)$.
    Let $\v{x}^* = S(\v{z}^*)$. 
    The weak convergence of $(\v{z}^\nu)$ implies that $(\v{z}^\nu)$ is bounded by \cite[Proposition 2.50]{bauschke_combettes}.
    By Lemma \ref{cab:Lemma:boundedness_coupled}, the boundedness of $(\v{z}^\nu)$ implies the boundedness of $(\v{x}^\nu)$, and therefore $(\v{x}^\nu)$ has a weak sequential cluster point \cite[Lemma 2.45]{bauschke_combettes}. %
    Let $\tilde{\v{x}}$ be a weak sequential cluster point of $(\v{x}^\nu)$, and $(\v{x}^{\nu_n})$ be a sequence that converges weakly to $\tilde{\v{x}}$.
    Since $ \MA \tilde{\v{x}} = \lim_{n \to \infty} \MA \v{x}^{\nu_n} = \lim_{\nu \to \infty} \MA \v{x}^\nu = 0$, we know $\v{\tilde{x}}$ is in the null space of $\MA$, which is $\mathcal{N}_A$.

    We now show that $\v{\tilde{x}} = \v{x}^*$.
    This rests on the maximal monotonicity of $\v{C} = \v{A} + \QA\v{B}\KA + \frac{1}{4}\UA$, which we have by the maximal monotonicity of $\v{A}$, Lemma \ref{cab:Lemma:PBK_coupled}, and the fact that the domain of $\QA\v{B}\KA + \frac{1}{4}\UA$ is $\Hx$, by \cite[Corollary 25.5]{bauschke_combettes}.
    By definition, we know that 
    \begin{equation*}
    \v{x}^{\nu_n} = J_{\alpha\DA^{-1}\v{A}}\left(\DA^{-1}\left[-\MT \v{z}^{\nu_n} - \alpha \QA\v{B}\KA \v{x}^{\nu_n} + 2\LA \v{x}^{\nu_n}\right]\right).
    \end{equation*}
    Therefore
    \begin{align*}
        \v{x}^{\nu_n} + \alpha\DA^{-1}\v{A}\v{x}^{\nu_n} &\ni \DA^{-1}\left[-\MT \v{z}^{\nu_n} - \alpha \QA\v{B}\KA \v{x}^{\nu_n} + 2\LA \v{x}^{\nu_n}\right] \\
        \alpha \left(\v{A}\v{x}^{\nu_n} + \QA\v{B}\KA \v{x}^{\nu_n}\right) &\ni -\MT \v{z}^{\nu_n} + (2\LA - \DA) \v{x}^{\nu_n} .
    \end{align*}
    Adding $\frac{\alpha}{4}\UA \v{x}^{\nu_n}$ to both sides, we get
        \[\alpha\v{C}\v{x}^{\nu_n} \ni -\MT \v{z}^{\nu_n} + \left(\frac{\alpha}{4}\UA + 2\LA - \DA\right) \v{x}^{\nu_n} .\]
    Let $\v{v} \in \alpha\v{C}\v{u}$ for some $\v{u} \in \Hx$.
    By the monotonicity of $\v{C}$, we have
    \begin{subequations}
    \begin{align}
        \left\langle \v{x}^{\nu_n} - \v{u}, -\MT \v{z}^{\nu_n} + \left(\frac{\alpha}{4}\UA + 2\LA - \DA\right) \v{x}^{\nu_n} - \v{v} \right\rangle &\geq 0 
        \label{cab:cluster_point_1_coupled}\\
        \left\langle \v{x}^{\nu_n}, -\MT \v{z}^{\nu_n}\right\rangle + \frac{\alpha}{4}\left\langle \v{x}^{\nu_n}, \UA \v{x}^{\nu_n} \right\rangle - \left\langle \v{x}^{\nu_n}, \ZA\v{x}^{\nu_n} \right\rangle & \nonumber\\
        - \left\langle \v{u}, -\MT \v{z}^{\nu_n} + \left(\frac{\alpha}{4}\UA + 2\LA - \DA\right)\v{x}^{\nu_n} - \v{v} \right\rangle &\geq 0. \label{cab:innerproducts}
    \end{align}
    \end{subequations}
    By \cite[Lemma 2.51(iii)]{bauschke_combettes}, since $\MA\v{x}^\nu \to 0$ and $\v{z}^\nu \weakto \v{z}^*$, we have %
    \begin{equation*}
        \left\langle \v{x}^{\nu_n}, -\MT \v{z}^{\nu_n}\right\rangle = \left\langle -\MA\v{x}^{\nu_n}, \v{z}^{\nu_n}\right\rangle \to 0.
    \end{equation*}
    We also know that $\MA\v{\tilde{x}} = 0$, so $\langle \v{x}^{\nu_n}, -\MT \v{z}^{\nu_n}\rangle \to \langle \v{\tilde{x}}, -\MT \v{z}^* \rangle$.
    By \eqref{cab:cabra_reqs}, we know that each $Z_k$ is bounded and PSD, and shares a null space with $W_k$, so 
    \begin{align*}
        \left\langle \v{x}^{\nu_n}, \ZA\v{x}^{\nu_n} \right\rangle &\leq \frac{\max_{k \in \mathcal{K}}\lambda_{\text{max}}(Z_k)}{\min_{k \in \mathcal{K}}\lambda_2(W_k)}\left\langle \v{x}^{\nu_n}, \WA\v{x}^{\nu_n} \right\rangle \\
        &\leq \frac{\max_{k \in \mathcal{K}}\lambda_{\text{max}}(Z_k)}{\min_{k \in \mathcal{K}}\lambda_2(W_k)}\norm{\MA \v{x}^{\nu_n}}^2 \to 0
    \end{align*}
    and $\langle \v{x}^{\nu_n}, \ZA\v{x}^{\nu_n} \rangle \to 0 = \langle \v{\tilde{x}}, (\DA - 2\LA) \v{\tilde{x}} \rangle$.
    Similarly, since by \eqref{cab:cabra_Z_succeq_U} $Z_k \succeq U_k$ and $\nullspace(U_k) \subseteq \nullspace{Z_k}$ for all $k \in \mathcal{K}$, we have $\langle \v{x}^{\nu_n}, \UA\v{x}^{\nu_n} \rangle \to 0 = \langle \v{\tilde{x}}, \UA\v{\tilde{x}} \rangle$.
    The remaining inner product on $\v{u}$ in \eqref{cab:innerproducts} converges by the weak convergence of $\v{x}^{\nu_n}$ and $\v{z}^{\nu_n}$, so we have 
    \begin{align*}
        &\left\langle \v{x}^{\nu_n} - \v{u}, -\MT \v{z}^{\nu_n} + \left(\frac{\alpha}{4}\UA + 2\LA - \DA\right) \v{x}^{\nu_n} - \v{v} \right\rangle \\
        \to &\left\langle \v{\tilde{x}} - \v{u}, -\MT \v{z}^* + \left(\frac{\alpha}{4}\UA + 2\LA - \DA\right) \v{\tilde{x}} - \v{v} \right\rangle
    \end{align*}
    and by \eqref{cab:cluster_point_1_coupled}
    \begin{equation*}
        \left\langle \v{\tilde{x}} - \v{u}, -\MT \v{z}^* + \left(\frac{\alpha}{4}\UA + 2\LA - \DA\right) \v{\tilde{x}} - \v{v} \right\rangle \geq 0.
    \end{equation*}
    The maximality of $\alpha\v{C}$ then requires that 
    \[-\MT \v{z}^* + \left(\frac{\alpha}{4}\UA + 2\LA - \DA\right) \v{\tilde{x}} \in \alpha\v{C}(\v{\tilde{x}})\]
    and therefore
    \begin{align*}
            \alpha \left(\v{A} + \QA\v{B}\KA + \frac{1}{4}\UA\right) \v{\tilde{x}} &\ni -\MT \v{z}^* + \left(\frac{\alpha}{4}\UA + 2\LA - \DA\right) \v{\tilde{x}}\\
            \alpha \left(\v{A} + \QA\v{B}\KA \right) \v{\tilde{x}} &\ni -\MT \v{z}^* + \left(2\LA - \DA\right) \v{\tilde{x}}\\
            \DA\v{\tilde{x}} + \alpha \v{A}\v{\tilde{x}} &\ni -\MT \v{z}^* + 2\LA \v{\tilde{x}} - \alpha\QA\v{B}\KA\v{\tilde{x}}.
    \end{align*}
    This, in turn, means that 
    \begin{equation*}
        \v{\tilde{x}} = J_{\alpha\DA^{-1}\v{A}}\left(\DA^{-1}\left[ -\MT \v{z}^* - \alpha \QA\v{B}\KA \v{\tilde{x}} + 2\LA \v{\tilde{x}}\right]\right).
    \end{equation*}
    Therefore, by the strict block lower triangularity of $\LA$ and $\QA \v{B} \KA$ and the single-valuedness of the resolvent of a maximal monotone operator, the weak sequential cluster point $\v{\tilde{x}}$ is unique, and $\v{\tilde{x}} = \v{x}^*$ \cite{minty1962monotone}.
    Since $(\v{x}^\nu)$ has a unique weak sequential cluster point given by $\v{x}^*$, we know that $\v{x}^\nu \weakto \v{x}^*$, and therefore corresponds to a solution of \eqref{cab:zero_in_coupled_sum}.

\end{proof}

We now establish the convergence rate associated with CABRA.
\begin{theorem}
    \label{cab:convergence_rate}
    Let $(\v{z}^\nu)$ be the sequence of iterates produced by Algorithm \ref{cab:cabra_d_alg} with a sequence of step sizes $(\gamma_\nu) \in (0, 2 - \frac{\alpha l}{2})$ with $\sum_{\nu} \gamma_\nu (2 - \frac{\alpha l }{2} - \gamma_\nu) = +\infty$ applied to problem \eqref{cab:zero_in_coupled_sum}, $(\v{x}^\nu)$ be the corresponding sequence of solutions such that
    \begin{align*}
    \v{x}^\nu &= S(\v{z}^\nu), \\
    \bar{y}^\nu_{k} &= \frac{1}{n_k}\sum_{i \in \mathcal{I}_k} \v{x}_{ik}^\nu, \\
    \v{b}^\nu &= \v{B} \KA\v{x}^\nu, \;\text{and} \\
    \v{w}^\nu &= \frac{1}{\alpha}\left(-\MT \v{z}^\nu + (2\LA - \DA )\v{x}^\nu \right) - \QA \v{b}^\nu \in \v{A}\v{x}^\nu.
    \end{align*}
If the set of solutions of \eqref{cab:zero_in_coupled_sum} is non-empty, then 
\begin{enumerate}
    \item $\begin{aligned}
    \sum_{k \in \mathcal{K}}\sum_{i \in \mathcal{I}_k}\norm{\v{x}_{ik}^\nu - \bar{y}_{k}^\nu}^2 = o\left(\frac{1}{\nu +1}\right)
    \end{aligned}$
    \item $\begin{aligned}
        \norm{\sum_{i \in \mathcal{I}}\RAi{i}^* \v{w}_i^\nu + \sum_{j \in \mathcal{J}} \left(\RBj{j}\right)^* \v{b}_j^\nu}^2 = o\left(\frac{1}{\nu +1}\right).
        \end{aligned}$
\end{enumerate}
\end{theorem}
\begin{proof}
    By \cite[Theorem 1]{davis2017convergence}, since $T$ is averaged non-expansive and our step sizes meet the regularity conditions, we have $\norm{\MA\v{x}^\nu}^2 = o(\frac{1}{\nu + 1})$.
    Let $\v{y}^\nu = \PAs \v{x}^\nu$.
    Then \[\MA \v{x}^\nu = \v{M} \v{y}^\nu = (\v{M}_k \v{y}_k^\nu)_{k \in \mathcal{K}}.\]
    Let $\bar{\v{y}}_k^\nu = (\bar{y}_k^\nu)_{i \in \mathcal{I}_k} \in \GG_k^{n_k}$ give a lifted vector with $n_k$ copies of the average value of the entries of $\v{y}^\nu_k$.
    Note that since the nullspace of each $W_k$ is spanned by the ones vector and $W_k = M_k^\intercal M_k$, we have 
    \begin{align*}
        \lambda_2(W_k)\norm{\v{y}_k^\nu - \bar{\v{y}}_k^\nu}^2 &\leq \left\langle \v{W}_k \left(\v{y}_k^\nu - \bar{\v{y}}_k^\nu \right), \v{y}_k^\nu - \bar{\v{y}}_k^\nu\right\rangle \\
        &= \norm{\v{M}_k \left(\v{y}_k^\nu - \bar{\v{y}}_k^\nu\right)}^2 = \norm{\v{M}_k \v{y}_k^\nu}^2.
    \end{align*}
    Therefore $\sum_{k \in \mathcal{K}} \lambda_2(W_k)\norm{\v{y}_k^\nu - \bar{\v{y}}_k^\nu}^2 = o(\frac{1}{\nu + 1})$.
    By the definition of the product spaces and permutation, 
    \begin{align*}
    \sum_{k \in \mathcal{K}}\norm{\v{y}_k^\nu - \bar{\v{y}}_k^\nu}^2 &= \sum_{k \in \mathcal{K}}\sum_{i \in \mathcal{I}_k}\norm{\v{y}_{ki}^\nu - \bar{y}_{k}^\nu}^2 \\
    &= \sum_{k \in \mathcal{K}}\sum_{i \in \mathcal{I}_k}\norm{\v{x}_{ik}^\nu - \bar{y}_{k}^\nu}^2, 
    \end{align*}
    proving the first result.
    
    For the second result, we note that in iteration $\nu$ we have
    \begin{equation*}
        \DA \v{x}^\nu + \alpha \v{A} \v{x}^\nu \ni -\MT \v{z}^\nu + 2\LA \v{x}^\nu - \alpha \QA \v{B} \KA \v{x}^\nu
    \end{equation*}
    so that, using $\v{w}^\nu$ and $\v{b}^\nu$ as defined above, we have
    \begin{equation*}
        \alpha \v{w}^\nu + \alpha \QA \v{b}^\nu = -\MT \v{z}^\nu + (2\LA - \DA)\v{x}^\nu. 
    \end{equation*}
    Permuting this, we have
    \begin{equation*}
        \alpha \PAs \v{w}^\nu + \alpha \v{Q} \PBs \v{b}^\nu = -\v{M}^* \v{z}^\nu + (2\v{L}- \v{D})\PAs\v{x}^\nu. 
    \end{equation*}
    Letting $\v{y}^\nu = \PAs \v{x}^\nu$, and summing over the entries associated with each $k$, 
    the $-\v{M}^* \v{z}^\nu$ term disappears since, by \eqref{cab:M_connect}, $\1^\intercal M_k^\intercal = 0$, and we have
    \begin{align*}
        \alpha \sum_{i \in \mathcal{I}_k} \left(\v{w}_{ik} + \sum_{j \in \mathcal{J}_k} (\v{Q}_k)_{ij} \v{b}_{jk} \right) &= (\1_{n_k}^{\intercal} (2 L_k - D_k) \otimes \I_{k})\v{y}_k^\nu \\
        \alpha \left(\sum_{i \in \mathcal{I}_k}\v{w}_{ik} + \sum_{j \in \mathcal{J}_k} \v{b}_{jk}\right)&= (\1_{n_k}^{\intercal} (2 L_k - D_k) \otimes \I_{k})\v{y}_k^\nu.
    \end{align*}
    The second equality derives from the requirement that $\1_{n_k}^\intercal Q_k = \1_{m_k}$.
    Let $u_k = \1_{n_k}^\intercal (2 L_k - D_k)$ and note that $ u_k \1_{n_k}= 0$ by \eqref{cab:DL_sum}.
    Therefore 
    \begin{align*}
    \norm{ (u_k^\intercal \otimes \I_k) \v{y}_k^\nu}^2 &= \norm{ (u_k^\intercal \otimes \I_k )\left(\v{y}_k^\nu - \bar{\v{y}}_k^\nu \right)}^2 \\
    &\leq \norm{u_k}^2\norm{\v{y}_k^\nu - \bar{\v{y}}_k^\nu}^2.
    \end{align*}
    Summing this inequality over all $k \in \mathcal{K}$, we obtain
    \begin{equation*}
        \alpha^2 \sum_{k \in \mathcal{K}} \norm{\sum_{i \in \mathcal{I}_k} \v{w}_{ik} + \sum_{j \in \mathcal{J}_k} \v{b}_{jk} }^2 \leq \sum_{k \in \mathcal{K}} \norm{u_k}^2 \norm{\v{y}_k^\nu - \bar{\v{y}}_k^\nu}^2 \leq \left(\max_{k \in \mathcal{K}} \norm{u_k}^2 \right) \sum_{k \in \mathcal{K}} \norm{\v{y}_k^\nu - \bar{\v{y}}_k^\nu}^2.
    \end{equation*}
    By definition, 
    \begin{equation*}
        \norm{\sum_{i \in \mathcal{I}}(\RAi{i})^* \v{w}_i^\nu + \sum_{j \in \mathcal{J}} (\RBj{j})^* \v{b}_j^\nu}^2 = \sum_{k \in \mathcal{K}} \norm{\sum_{i \in \mathcal{I}_k} \v{w}_{ik} + \sum_{j \in \mathcal{J}_k} \v{b}_{jk}
        }^2 .
    \end{equation*}
    Therefore, since $\alpha > 0$ and $\max_{k \in \mathcal{K}} \norm{u_k}^2 < \infty$, applying the first result yields 
    \begin{equation*}
        \norm{\sum_{i \in \mathcal{I}}(\RAi{i})^* \v{w}_i^\nu + \sum_{j \in \mathcal{J}} (\RBj{j})^* \v{b}_j^\nu}^2 = o\left(\frac{1}{\nu+1}\right).
    \end{equation*}
\end{proof}

\subsection*{Equivalence of Expanded and Reduced Algorithms}
Algorithm \ref{cab:cabra_n_alg} gives the expanded version of the CABRA applied to problem \eqref{cab:zero_in_coupled_sum}.
We note the equivalence of Algorithm \ref{cab:cabra_d_alg} and Algorithm \ref{cab:cabra_n_alg} in the following lemma, with $M_k^\dagger$ giving the Moore-Penrose pseudo-inverse of $M_k^\intercal$, and $\MA^\dagger$ constructed in the same manner as $\MT$ from $M_k^\dagger$.
    
\begin{lemma}
    \label{cab:Lemma:expanded_equivalent}
    With $\v{v}^0 = -\MT \v{z}^0$ and a given sequence of step sizes, Algorithm \ref{cab:cabra_n_alg} produces the same set of iterates $(\v{x}^\nu)$ as Algorithm \ref{cab:cabra_d_alg}, with $(\v{v}^\nu) = (-\MT \v{z}^\nu)$. Similarly, with $\v{z}^0 = -\MA^\dagger \v{v}^0$ and any given sequence of step sizes, Algorithm \ref{cab:cabra_d_alg} produces the same set of iterates $(\v{x}^\nu)$ as Algorithm \ref{cab:cabra_n_alg}, with $(\v{z}^\nu) = (-\MA^\dagger \v{v}^\nu)$.
\end{lemma}
\begin{proof}
    With the substitution $\v{v}^\nu = -\MT \v{z}^\nu$, step \eqref{cab:cabra_d_itr2} in Algorithm \ref{cab:cabra_d_alg} becomes step \eqref{cab:cabra_n_itr2} in Algorithm \ref{cab:cabra_n_alg}, with a similar correspondence in step \eqref{cab:cabra_d_itr1}. Because each $A_i$ is maximal monotone, its resolvent is unique, proving the first result.
    We prove the second result by showing that 
    \begin{align*}
    \v{x}^\nu &= J_{\alpha \DA^{-1}\v{A}}\left(\DA^{-1}\left[-\MT \v{z}^\nu + 2\LA \v{x}^\nu - \alpha \QA\v{B}\KA\v{x}^\nu\right]\right) \\
    &= J_{\alpha \DA^{-1}\v{A}}\left(\DA^{-1}\left[\v{v}^\nu + 2\LA \v{x}^\nu - \alpha \QA\v{B}\KA\v{x}^\nu\right]\right),
    \end{align*}
    and that whenever we have $\v{z}^\nu =  -\MA^\dagger \v{v}^\nu$, we must therefore have $\v{z}^{\nu + 1} =  -\MA^\dagger \v{v}^{\nu + 1}$.
    The equality of the $\v{x}^\nu$ values follows directly from the fact that $-\MT \v{z}^\nu = -\MT \left(-\MA^\dagger \v{v}^\nu \right) = \v{v}^\nu$.
    Furthermore, for a given update $\v{x}^\nu$, we must have $\v{z}^{\nu+1} =  -\MA^\dagger \v{v}^{\nu + 1}$ since  
    \begin{align*}
   \v{z}^{\nu+1} &= \v{z}^\nu + \gamma_\nu \MA \v{x}^\nu \\
    &= -\MA^\dagger \v{v}^\nu - \gamma_\nu \MA \v{x}^\nu \\
    &= -\MA^\dagger \left(\v{v}^\nu + \gamma_\nu \MT \MA \v{x}^\nu \right)\\
    &= -\MA^\dagger \left(\v{v}^\nu + \gamma_\nu \WA \v{x}^\nu \right)\\
    &= -\MA^\dagger \v{v}^{\nu+1}.
    \end{align*}
    Thus, by induction, Algorithm \ref{cab:cabra_d_alg} produces the same set of iterates $(\v{x}^\nu)$ as Algorithm \ref{cab:cabra_n_alg}, with $(\v{z}^\nu) = (-\MA^\dagger \v{v}^\nu)$.

\end{proof}
This expanded form of CABRA provides a helpful insight into the properties of its fixed points which we can use to warm start the algorithm.
At a fixed point of $T$, we have \[\DA \v{x} + \alpha\v{w} = \v{v} - \alpha \QA \v{B}\xb + 2 \LA\v{x}\] for some $y^*$ in the solution set of problem \eqref{cab:zero_in_coupled_sum}, $\v{x} = \RA y^*$, $\xb = \RB y^*$, and $\v{w} \in \v{A}\v{x}$.
Therefore, if we have an estimate $\tilde{y}$ of $y^*$, we can choose a warm start value of $\v{v}^0 = (\DA - 2\LA)\RA \tilde{y} $.
We can also augment this with an estimate of $\v{w}^*$ and $\QA \v{B} \xb$ as long as we project that estimate onto $\mathcal{N}_A^\perp$ prior to scaling it by $\alpha$ and adding it.
This projection is required because we must have $\v{v}^0 \in \range(\MT)$, and an arbitrary estimate $\tilde{\v{w}}$ and $\QA \v{B}\tilde{\v{x}}'$ may not satisfy $\tilde{\v{w}} + \QA \v{B}\tilde{\v{x}}'\in \range(\MT)$.

We conclude with a final lemma on the boundedness of the iterates $\v{x}^\nu$, which we use in Theorem \ref{cab:main_theorem}.
\begin{lemma}    \label{cab:Lemma:boundedness_coupled}
    For any bounded sequence $(\v{z}^\nu)$ given by $\v{z}^{k+1} = T(\v{z}^\nu)$, its corresponding sequence $(\v{x}^\nu)$ given by \eqref{cab:cabra_d_itr1} is also bounded.
\end{lemma}
\begin{proof}
    This can be shown by induction over the sequences $(x_i^\nu)$ of $(\v{x}^\nu)$.
Let $x_i^* = \RAi{i} y^*$ be the component-wise solutions associated with a solution $y^*$ of problem \eqref{cab:zero_in_coupled_sum}, and $w_i^* \in A_i x_i^*$ be a corresponding set of operator values.
By the uniqueness of the resolvent of the maximal monotone operator $A_i$, we have
\[J_{\alpha\DA^{-1}_i A_i}\left( x_i^* + \alpha\DA^{-1}_i w_i^*\right) = x^*_i.\]
For any point $u_i$, by the triangle inequality and the nonexpansivity of the resolvent of a maximal monotone operator, we then have
\begin{align*}
    \norm{J_{\alpha\DA^{-1}_i A_i}(u_i)} &= \norm{J_{\alpha\DA^{-1}_i A_i}(u_i) - x_i^* + x_i^*} \\
    &\leq \norm{J_{\alpha\DA^{-1}_i A_i}(u_i) - J_{\alpha\DA^{-1}_i A_i}\left( x_i^* + \alpha\DA^{-1}_i w_i^*\right) + x_i^*} \\
    &\leq \norm{u_i - \left(x^*_i + \alpha\DA^{-1}_i w_i^*\right)} + \norm{x_i^*}.
\end{align*}
Let $u^\nu_1 = -(\DA^{-1})_1 (\MT \v{z}^\nu)_1$. 
Then
\begin{align*}
    \norm{x_1^\nu} &= \norm{J_{\alpha(\DA^{-1})_1 A_1}(u^\nu_1)} \\
                 &\leq \norm{u^\nu_1 - \left(x^*_1 + \alpha(\DA^{-1})_1 w_1^*\right)} + \norm{x_1^*} \\
                 &\leq \norm{(\DA^{-1})_1 (\MT \v{z}^\nu)_1} + \norm{x^*_1 + \alpha(\DA^{-1})_1 w_1^*} + \norm{x_1^*}
\end{align*}
which is bounded by the boundedness of $(\v{z}^\nu)$, $\DA^{-1}$, and $\MA$.

Suppose $(x_d^\nu)$ is bounded for $d < i$.
We now show that $(x_i^\nu)$ is bounded.
By design, $(\LA \v{x})_i$ and $(\QA \v{B} \KA \v{x})_i$ depend only on values $x_d^\nu$ for $d < i$.
Therefore $(\LA\v{x}^\nu)_i$ and $(\KA\v{x}^\nu)_j$ for all $Q_{ij} \neq 0$ are bounded by the induction hypothesis and the boundedness of $\LA$ and $\KA$, and $(\QA\v{B}\KA\v{x}^\nu)_i$ is bounded by the boundedness of $\QA$ and the cocoercivity of $\v{B}$.
The boundedness of $(\v{z}^\nu)$ implies the boundedness of $(\MT \v{z}^\nu)_{i}$.
Let $u^\nu_i = \DA^{-1}_i\left[-(\MT \v{z}^\nu)_{i} - \alpha(\QA\v{B}\KA\v{x}^\nu)_i + (2\LA\v{x}^\nu)_i\right]$. 
We therefore have
\begin{align*}
    \norm{x_i^\nu} &= \norm{J_{\alpha\DA^{-1}_i A_i}(u^\nu_i)} \\
    &\leq \norm{u^\nu_i - \left(x^*_i + \alpha\DA^{-1}_i w_i^*\right)} + \norm{x_i^*} \\
    &\leq \norm{\DA^{-1}_i (\MT \v{z}^\nu)_{i}} + \alpha\norm{\DA^{-1}_i (\QA\v{B}\KA\v{x}^\nu)_i} + \norm{\DA^{-1}_i (2\LA\v{x}^\nu)_i} + C_i,
\end{align*}
with 
\[C_i = \norm{x^*_i + \alpha\DA^{-1}_i w_i^*} + \norm{x_i^*}\]
and $\norm{x_i^\nu}$ is bounded for all $\nu$.
Therefore, by induction, $(\v{x}_i^\nu)$ is bounded for all $i \in \mathcal{I}$, and $(\v{x}^\nu)$ is bounded.
\end{proof}

\end{document}